\documentclass[12pt,A4,reqno]{amsart}
\usepackage{amsfonts}
\usepackage{mathrsfs}
\usepackage[T1]{fontenc}
\usepackage{amssymb}
\usepackage{amsmath,amscd}
\usepackage{color}
\usepackage{epsf}
\usepackage{graphicx}
\usepackage{extarrows}
\usepackage[curve]{xypic}

\theoremstyle{plain}
\newtheorem{thm}{Theorem}[section]
\newtheorem{lem}[thm]{Lemma}
\newtheorem{prop}[thm]{Proposition}
\newtheorem{cor}[thm]{Corollary}

\theoremstyle{definition}
\newtheorem{defi}[thm]{Definition}
\newtheorem{exam}[thm]{Example}
\newtheorem{rem}[thm]{Remark}

\newcommand{\R}{\mathbb R}
\newcommand{\Z}{\mathbb Z}

\newcommand{\nn}{\vskip 0.2cm}
\newcommand{\n}{\vskip 0.1cm}

\begin{document}

\title [\ ] {Fundamental groups of small covers revisited}

\author{Lisu Wu}
\address{Department of Mathematics, Nanjing University, Nanjing, 210093, P.R.China.
  }
 \email{wulisuwulisu@qq.com}

\author{Li Yu}
\address{Department of Mathematics and IMS, Nanjing University, Nanjing, 210093, P.R.China.
  }
 \email{yuli@nju.edu.cn}

\date{\today}

%\keywords{small cover, finite group actions}

%\subjclass[2000]{57M60, 57M50, 57S17, 57R85}

\thanks{2010 \textit{Mathematics Subject Classification}. 57N16, 57S17, 57S25\\
 This work is partially supported by 
 Natural Science Foundation of China (grant no.11371188) and 
 the PAPD (priority academic program development) of Jiangsu higher education institutions.}

\begin{abstract}
 We study the topology of small covers from their fundamental groups.
 We find a way to obtain explicit presentations of the 
 fundamental group of a small cover. Then we use these presentations to
 study the relations between the fundamental groups of a small cover and its facial submanifolds. In particular, we can
 determine when a facial submanifold of a small cover
  is $\pi_1$-injective
 in terms of some purely combinatorial data on the underlying simple polytope. 
 In addition,
 we find that any $3$-dimensional small cover has
         an embedded non-simply-connected $\pi_1$-injective surface. Using this result and some results of Schoen and Yau, we characterize all the $3$-dimensional small covers that admit
          Riemannian metrics with nonnegative scalar curvature.
  \end{abstract}

\maketitle

 \section{Introduction}
 
  The notion of small cover is first introduced by
  Davis and Januszkiewicz~\cite{DaJan91} as
  an analogue of a smooth projective toric variety in the category of closed manifolds
  with $\Z_2$-torus actions ($\Z_2 =\Z\slash 2\Z$).
 An $n$-dimensional \emph{small cover} is a closed $n$-manifold
  $M$ with a locally standard $(\Z_2)^n$-action whose orbit space can be identified with
  a simple convex polytope $P$ in an Euclidean space. A polytope is called \emph{simple} if every codimension-$k$ face is the intersection of exactly $k$ distinct \emph{facets} (codimension-one faces) of the polytope.
  Recall that two convex polytopes are combinatorially
equivalent if there exists a bijection between their posets of faces with respect to the inclusion.
   All polytopes considered in this paper
  are convex, so we omit the word ``convex'' for brevity. And in most cases we make no distinction
between a convex polytope and its combinatorial equivalence class. \n
  
   The $(\Z_2)^n$-action on the small cover $M$ determines a $(\Z_2)^n$-valued \emph{characteristic
       function} $\lambda$ on the set of \emph{facets} of
      $P$, which encodes all the information of the isotropy groups of
      the non-free orbits. Indeed for any facet $F$ of $P$, 
       the rank-one subgroup $\langle\lambda(F)\rangle \subset (\Z_2)^n$ 
    generated by $\lambda(F)$ 
     is the isotropy group of the codimension-one
      submanifold $\pi^{-1}(F)$ of $M$ where $\pi: M\rightarrow P$
     is the orbit map of the $(\Z_2)^n$-action. The function $\lambda$ is 
     \emph{non-degenerate} in the sense that the values of $\lambda$ 
     on any $n$ facets that are incident to a vertex of $P$ form 
     a basis of $(\Z_2)^n$.
     Conversely, we can recover the manifold $M$
       by gluing $2^n$ copies of $P$ according to the function $\lambda$.
     For any proper face $f$ of $P$, define
      \begin{equation} \label{Equ:G_f}
         G_f = \text{the subgroup of $(\Z_2)^n$ generated by
      the set $\{ \lambda(F) \, |\, f\subset F \}$}.
      \end{equation}
      And define $G_P =\{0\} \subset (\Z_2)^n$.
      Then $M$ is homeomorphic to the quotient space
    \begin{equation} \label{Equ:Quo-SC}
      P\times (\Z_2)^n \slash \sim
    \end{equation}  
  where $(p,g) \sim (p',g')$ if and only if $p=p'$ and $g^{-1}g' \in G_{f(p)}$, and
  $f(p)$ is the unique face of $P$ that contains $p$ in its relative interior. Let $\Theta$ be
  the quotient map in~\eqref{Equ:Quo-SC}. 
   \begin{equation} \label{Equ:Theta}
      \Theta:   P\times (\Z_2)^n  \rightarrow  P\times (\Z_2)^n \slash \sim. 
    \end{equation} 
 \n

     It is shown in~\cite{DaJan91} that many important topological invariants of $M$
     can be easily computed in terms of the combinatorial
      structure of $P$ and the characteristic 
      function $\lambda$. In particular, we can determine
      the fundamental group $\pi_1(M)$ of $M$ as follows.
      Let $W_{P}$ be a right-angled Coxeter group with one generator $s_F$
      and relations $s_F^2=1$ for each facet $F$ of $P$, and 
     $(s_Fs_{F'})^2 =1$ whenever
      $F,F'$ are adjacent facets of $P$.    
      Note that if $F\cap F'=\varnothing$, 
       $s_{F}s_{F'}$ has infinite order in $W_P$ 
       (see~\cite[Proposition 1.1.1]{BjorBre05}).          
      According to~\cite[Lemma 4.4]{DaJan91}, $W_{P}$ is isomorphic to the
      fundamental group of the Borel construction $M_{(\Z_2)^n}=E(\Z_2)^n \times_{(\Z_2)^n} M$
      of $M$.
         It is shown in~\cite[Corollary 4.5]{DaJan91} that the homotopy exact sequence
         of the fibration
         $M \rightarrow M_{(\Z_2)^n} \rightarrow B(\Z_2)^n$ gives a short exact sequence
        \begin{equation} \label{Equ:Fund-Group}
           1 \longrightarrow \pi_1(M) \overset{\psi}{\longrightarrow} 
           W_{P} \overset{\phi}
           {\longrightarrow}
           (\Z_2)^n \longrightarrow 1,
          \end{equation}
         where $\phi(s_F) =\lambda(F)$ for any facet $F$ of $P$, and $\psi$
         is induced by the canonical map 
         $M \hookrightarrow M\times E(\Z_2)^n \rightarrow
         M_{(\Z_2)^n}$. 
         Hence $\pi_1(M)$ is isomorphic to the kernel of $\phi$.
         It follows that a small cover $M$ is never simply connected.
         Moreover, the sequence~\eqref{Equ:Fund-Group} actually splits, so
         $W_P$ is isomorphic to a semidirect product of $\pi_1(M)$ and $(\Z_2)^n$.
         \n
       
          Let $\mathcal{F}(P)$ denote the set of facets of $P$.
           For any proper face $f$ of $P$, we have the following definitions.
         \begin{itemize}

          \item  Define $\mathcal{F}(f^{\perp}) = \{ F\in \mathcal{F}(P) \, |\, 
           \dim(f\cap F) =\dim(f)-1\}$. In other words,
            $\mathcal{F}(f^{\perp})$ consists of those facets of $P$ 
             that intersect $f$ transversely. \n
             
          \item We call $M_f=\pi^{-1}(f)$ the \emph{facial submanifold} of $M$ corresponding to $f$. 
          It is easy to see that $M_f$
          is a small cover over the simple polytope $f$, whose characteristic function, denoted by 
          $\lambda_f$, is 
          determined by $\lambda$ as follows. Let 
          $$\rho_f: (\Z_2)^n \rightarrow
           (\Z_2)^n\slash G_f \cong (\Z_2)^{\dim(f)}$$ 
           be the quotient homomorphism. Then we have 
            \begin{equation} \label{Equ:Quotient-Color}
               \lambda_f(f\cap F) =  \rho_f ( \lambda(F) ), \ \forall\, F\in \mathcal{F}(f^{\perp}).
             \end{equation}  
             
           \item  Similarly to $W_P$, we obtain two
          group homomorphisms for $W_f$:
           $$\psi_f: \pi_1(M_f)\rightarrow W_f,\ \ \phi_f: W_f \rightarrow (\Z_2)^{\dim(f)}.$$
           
            \item Let $i_f: f \rightarrow P$ and $j_f : M_f \rightarrow M$
          be the inclusion maps. Then $i_f$ induces a natural group homomorphism
          $(i_f)_* : W_f \rightarrow W_P$ which sends the generator $s_{f\cap F}$ of $W_f$
          to the generator $s_{F}$ of $W_P$ for any facet $F\in \mathcal{F}(f^{\perp})$. 
           It is easy to check that $(i_f)_*$ is well defined.
          \end{itemize}
        
           Then we have the following diagram. 
              \[ \xymatrix{
       1 \ar[r] & \pi_1(M_f) \ar[r]^{\ \ \psi_f} \ar[d]^{(j_f)_* \ \ \textbf{??}} & W_f 
       \ar[r]^{\phi_f\quad }\ar[d]^{(i_f)_*} &   
         (\Z_2)^{\dim(f)}  \ar[r] & 1 \\
       1 \ar[r] &  \pi_1(M)  \ar[r]^{\ \ \psi} & W_{P} \ar[r]^{\phi} &  
       (\Z_2)^n \ar[u]_{\rho_f} \ar[r] & 1 
           } \]
         
             One may expect that this diagram commutes, i.e.
              $(i_f)_*\circ \psi_f =  \psi\circ (j_f)_*$.
             But this is not true in general.
              Indeed, $(i_f)_*$ may not map
           $\mathrm{ker}(\phi_f)$ into $\mathrm{ker}(\phi)$ (see Example~\ref{Example-1}). 
           It seems to us that there is no canonical way to relate the two maps 
            $(j_f)_*$ and $(i_f)_*$, essentially because there is no canonical way to embed
            $(\Z_2)^{\dim(f)} = (\Z_2)^n\slash G_f$ into $(\Z_2)^n$.
           So to understand $(j_f)_* : \pi_1(M_f)\rightarrow \pi_1(M)$, 
           we need to study the fundamental groups of small covers by some other methods.
         \n
         
          \begin{exam} \label{Example-1}
       In Figure~\ref{p:Example}, we have a facet $F$ in a $3$-dimensional
          simple polytope $P$ along with a $(\Z_2)^3$-coloring,
           where $e_1,e_2,e_3$ is a basis of $(\Z_2)^3$.
       Then the Coxeter group $W_F$ is generated by $\{ s_{F\cap F_i} \}_{1\leq i \leq 4}$ and
        $(i_F)_*:W_F \rightarrow W_{P}$ sends $s_{F\cap F_i}$ to the generator $s_{F_i}$ of
        $W_{P}$. Then $(i_F)_* (s_{F\cap F_1}s_{F\cap F_3}^{-1} ) = 
        s_{F_1} s_{F_3}^{-1}$. But
       \begin{itemize}
        \item $\phi_F(F\cap F_1) = \phi_F(F\cap F_3)$, then
              $s_{F\cap F_1}s_{F\cap F_3}^{-1} \in \mathrm{ker}(\phi_F) \subset W_F$;
        \item $\phi(F_1)=e_1+e_3$, $\phi(F_3)=e_3$, then $s_{F_1} s_{F_3}^{-1} \notin 
         \mathrm{ker}(\phi) \subset W_P$.
        \end{itemize}
        So $(i_F)_*$ does not map $\mathrm{ker}(\phi_F)$ into $ \mathrm{ker}(\phi)$.
          \begin{figure}
          % Requires \usepackage{graphicx}
         \includegraphics[width=0.18\textwidth]{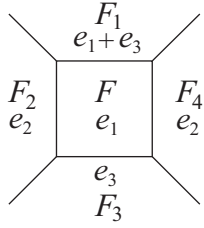}\\
          \caption{A $3$-dimensional simple polytope $P$}\label{p:Example}
      \end{figure} 
          
       \end{exam}

          A motive to study 
          $(j_f)_* : \pi_1(M_f)\rightarrow \pi_1(M)$ is to see 
          under what condition it is injective. A submanifold $\Sigma$ embedded 
          in $M$ is called \emph{$\pi_1$-injective}
           if the inclusion of $\Sigma$ into $M$
          induces a monomorphism in the fundamental group. The $\pi_1$-injective submanifolds play important roles
           in the studying of geometry and topology of low dimensional manifolds. For example the existence
           of $\pi_1$-injective immersed surfaces with nonpositive Euler characteristic in
            a compact oriented $3$-manifold 
           is an obstruction to the existence of Riemannian metric with 
           positive scalar curvature on the manifold (see~\cite{SchYau79}).
            In addition, if an irreducible orientable $3$-manifold
           has a non-simply-connected $\pi_1$-injective embedded surface, 
          then the $3$-manifold is a Haken manifold. It is shown in~\cite{Waldhau68} that
          Haken $3$-manifolds satisfy the \emph{Borel Conjecture}, 
          so they are determined up to homeomorphism by
           their fundamental groups. 
           The main result of this paper is the following theorem
           on the $\pi_1$-injectivity of a facial submanifold 
           in a small cover.\n
          
         \begin{thm}[Theorem~\ref{thm-Fund}.]
                  Let $M$ be a small cover over a simple polytope $P$ and 
      $f$ be a proper face of $P$. Then the following two statements are equivalent.
   \begin{itemize}
    \item[(i)] The facial submanifold $M_f$ is $\pi_1$-injective in $M$.\n
     \item[(ii)] For any $F,F'\in \mathcal{F}(f^{\perp})$, we have
            $f\cap F\cap F' \neq \varnothing$ whenever $F\cap F' \neq \varnothing$.
   \end{itemize}         
   \end{thm}
     
          This result is a bit surprising since it tells us that 
          the $\pi_1$-injectivity of $M_f$ in $M$ depends only on the combinatorics of the
          facets of $P$ adjacent to $f$. However the fundamental group of $M$ 
          depends not only on the combinatorial structure
          of $P$ but also on the characteristic function of $M$ in general.
          \n
           
           The main ingredient of our proof of Theorem~\ref{thm-Fund} is to use 
          explicit presentations of the fundamental groups of small covers. 
          We are also aware that
          methods from metric geometry can be used to prove (ii)$\Rightarrow$(i) in
          Theorem~\ref{thm-Fund} when $M$ is an aspherical manifold
          (see Proposition~\ref{Prop:Flag-Inject}
           and Remark~\ref{Rem:Flag-Metric}).
           By~\cite[Theorem 2.2.5]{DavJanScott98}, a 
           small cover $M$ over a simple polytope $P$ is 
           aspherical if and only if $P$ is \emph{flag} (i.e.
          a collection of facets of $P$ have common intersection 
          whenever they pairwise intersect).
          \n
            
          For any $n$-dimensional simple polytope $P$, let
           $P^*$ be the dual (or polar) polytope of $P$ (see~\cite[\S 2]{Ziegler95}). 
           Then its boundary $\partial P^*$ is a simplicial
          $(n-1)$-sphere. Any codimension-$k$ face $f$ of $P$ corresponds to a
          unique $(k-1)$-simplex in $\partial P^*$ denoted by $\sigma_f$. 
          If $f$ is the intersection of facets $F_1,\cdots, F_k$ of $P$,
          the vertices of $\sigma_f$ are $\sigma_{F_1},\cdots, \sigma_{F_k}$.
          Let $\mathrm{Lk}(\sigma_f,\partial P^*)$ be the link of
           $\sigma_f$ in $\partial P^*$.
          Then the vertices of $\mathrm{Lk}(\sigma_f,\partial P^*)$ are
           $\{ \sigma_F \,|\, F\in \mathcal{F}(f^{\perp})\}$.
           
           \begin{rem}
          The condition (ii) in Theorem~\ref{thm-Fund} is equivalent to saying that
          two vertices in $\mathrm{Lk}(\sigma_f,\partial P^*)$ are connected by
          an edge in $\mathrm{Lk}(\sigma_f,\partial P^*)$ if and only if they are
           connected by an edge in
          $\partial P^*$. But in general this condition does not imply that
           $\mathrm{Lk}(\sigma_f,\partial P^*)$ is a full subcomplex of 
          $\partial P^*$ (see~\cite[p.26]{BP02}). For example when $P$ is the $3$-simplex $\Delta^3$, any $2$-face $f$ of $P$ satisfies the condition (ii) in Theorem~\ref{thm-Fund} while $\mathrm{Lk}(\sigma_f,\partial P^*)$ is not a full subcomplex of
           $\partial P^*=\partial \Delta^3$.
          \end{rem}
           \n
          
          A \emph{$3$-belt} on a simple polytope $P$ consists of 
          three facets $F_1,F_2, F_3$ of $P$ which pairwise 
          intersect but have no common intersection.   
           \n
           
           \begin{cor} \label{Cor:3-belt}
            Let $M$ be a small cover over a simple polytope $P$. 
            For a facet $F$ of $P$,
            the facial submanifold $M_F$ is $\pi_1$-injective in $M$
             if and only if $F$ is not contained
            in any $3$-belt on $P$.
           \end{cor}
           
         In addition, we obtain the following description of 
         aspherical small covers in terms of the $\pi_1$-injectivity of their
         facial submanifolds.\n
         
      \begin{prop}[Proposition~\ref{Prop:Flag-Inject}]
           A small cover $M$ over a simple polytope $P$
            is aspherical if and only if all the facial submanifolds of 
            $M$ are $\pi_1$-injective in $M$.
    \end{prop}
           
         \begin{rem}
          For a small cover $M$, a facial submanifold $M_f$ is $\pi_1$-injective in $M$
           does not imply that $M_{f'}$ is $\pi_1$-injective in $M$ for any 
           $f'\subsetneq f$.
          Conversely, that $M_{f'}$ is $\pi_1$-injective in $M$ 
          for all $f'\subsetneq f$
          does not imply that $M_f$ is $\pi_1$-injective in $M$ either.\n
          \end{rem}

       \begin{exam} \label{Exam-Truncated-Prism}
         Let $M$ be any small cover over the $3$-dimensional 
         simple polytope $P_1$ or $P_2$ shown in
          Figure~\ref{p:Truncated-Prism}. 
            By Theorem~\ref{thm-Fund}, we can conclude the following. 
            
            \begin{itemize}
              \item For the triangular facet $F$ of $P_1$, the facial submanifold $M_F$
               is $\pi_1$-injective in $M$. 
               But for any $1$-face $f\subset F$, $M_f$ is not $\pi_1$-injective in $M$.\n 
               
               \item For the hexagonal facet $F'$ of $P_2$ whose vertex set is
               $\{v_1, v_2,v_3,v_4,v_5,v_6\}$,
                 the facial submanifold $M_{F'}$ is not $\pi_1$-injective 
                 in $M$. But for any face $f\subsetneq F'$, $M_f$ is
                 $\pi_1$-injective in $M$.
            \end{itemize}
       \end{exam}

      \begin{figure}
          % Requires \usepackage{graphicx}
         \includegraphics[width=0.57\textwidth]{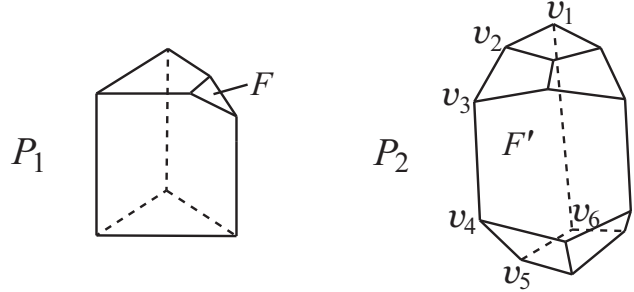}\\
          \caption{Examples of simple $3$-polytopes}\label{p:Truncated-Prism}
      \end{figure}
      
       Using Corollary~\ref{Cor:3-belt}, we can show
       that any $3$-dimensional small cover has a non-simply-connected embedded $\pi_1$-injective surface.\n

 \begin{prop}[Proposition~\ref{Prop:Facial-3-dim}]
     For any small cover $M$ over a $3$-dimensional simple polytope $P$, there
     always exists a facet $F$ of $P$ so that the facial submanifold $M_F$ is 
     $\pi_1$-injective in $M$.
 \end{prop}    
   
       By Proposition~\ref{Prop:Facial-3-dim} and some results in Schoen-Yau~\cite{SchYau79}, we can characterize
         all the $3$-dimensional small covers that admit
         Riemannian metrics with nonnegative scalar curvature. 
       \n
          
 \begin{prop}[Proposition~\ref{Prop:ScalCur-SmallCover}]
     A small cover $M$ over a simple $3$-polytope $P$ can hold a Riemannian metric with nonnegative
    scalar curvature if and only if $P$ is combinatorially equivalent to the cube $[0,1]^3$
   or a polytope obtained from $\Delta^3$ by a sequence of vertex cuts.
    In particular, all the orientable $3$-dimensional small covers that can hold Riemannian metrics with nonnegative scalar curvature are the two orientable real Bott manifolds in dimension $3$ and the
 connected sum of $k$ copies of $\R P^3$ for any $k\geq 1$.
 \end{prop}

         The paper is organized as follows. In section~\ref{Sec:Presentaion}, we 
         construct some explicit presentation
          of the fundamental group of a small cover $M$ over
         a simple polytope $P$ (see Proposition~\ref{Prop:Presentation}) 
          by constructing a special cell decomposition of $M$. 
         Using this presentation of $\pi_1(M)$ and the presentation of the right-angled 
         Coxeter group $W_P$, we
         construct an explicit monomorphism from $\pi_1(M)$ into $W_P$.
         In section~\ref{Sec:Facial-SubMfd}, we obtain 
         a necessary and sufficient condition for
         a facial submanifold $M_f$ of $M$ to be $\pi_1$-injective (Theorem~\ref{thm-Fund}).
         Moreover, we determine the kernel of the homomorphism
         $(j_f)_*: \pi_1(M_f,v) \rightarrow \pi_1(M,v)$ in Theorem~\ref{thm-Fund-Kernel}.
         In section~\ref{Sec:ScalarCurvature}, we show in Proposition~\ref{Prop:Facial-3-dim} that 
         any $3$-dimensional small cover has a non-simply-connected embedded $\pi_1$-injective surface. 
         Then we characterize all the $3$-dimensional small covers and
           real moment-angle manifolds that admit Riemannian metrics with nonnegative scalar curvature.
          
          \vskip .6cm

 \section{Presentations of the fundamental groups of small covers} \label{Sec:Presentaion}
   
   \subsection{Cell decompositions of small covers} \label{Subsec-Cell} \ \n
     To obtain a presentation of the fundamental group of a small cover $M$, we 
     can use the $2$-skeleton of a
     cell decomposition of $M$. In~\cite{DaJan91}
    two kinds of cell decompositions of small covers are constructed.
    \begin{itemize}
     \item[(I)] The first construction is to use a special type of Morse 
     functions on $P$. 
    The corresponding cell decompositions of the small cover $M$ 
     are ``perfect'' in the sense that
    their cellular chain complexes with $\Z_2$-coefficients
      have trivial boundary maps. 
      \item[(II)] The second construction is to use a 
       standard cubical subdivision of $P$ and decompose the small 
       cover $M$ as a union of big cubes
       via the gluing procedure in~\eqref{Equ:Quo-SC}.
      \end{itemize}
    
        But it is not so convenient for us to write
      a presentation of $\pi_1(M)$ from either of these cell decompositions.
      In the construction (I), the gradient flow of the Morse function on $M$ may
      flow from one critical point to another one with equal or higher index.
      So to obtain the representatives of the generators of $\pi_1(M)$, we need to make
      some modification of the flow lines which could be hard to handle. 
       In the construction (II), there are more than one $0$-cell in the decomposition. 
       So in reality we need to shrink a maximal tree in the $1$-skeleton to a point to
      obtain the representatives of the generators of $\pi_1(M)$. But there is no 
      natural choice
      of a maximal tree in the $1$-skeleton of this cell decomposition.
       In the following, we will slightly modify the
       construction (II) and 
      obtain a new cell decomposition of $M$ which has only one $0$-cell.\n
      
      Recall how to obtain a cubical subdivision of an $n$-dimensional
       simple polytope $P$ (see~\cite[Sec.1.2]{DavJanScott98}). First of all, the set of faces of $P$, 
       partially ordered by inclusion, forms a poset denoted by $\mathcal{P}$.
       Let $\mathrm{FL}(\mathcal{P})$ be the \emph{order complex} of $\mathcal{P}$ which 
       consists of all flags in $\mathcal{P}$. For any flag 
       $\alpha = (f_1 <\cdots <f_k) \in \mathrm{Fl}(\mathcal{P})$, let $|\alpha|=k$.\n
       
       For any $f, f' \in \mathcal{P}$ with $f \leq f'$, we have the following definitions.
        \begin{itemize}
          \item Denote by $[f,f']$ be the subposet 
          $\mathcal{P}_{\geq f} \cap \mathcal{P}_{\leq f'}$ of $\mathcal{P}$.\n
          
          \item Denote by $[f,f')$ be the subposet 
          $\mathcal{P}_{\geq f} \cap \mathcal{P}_{< f'}$ of $\mathcal{P}$.
        \end{itemize} 
        The order complex of $[f,f']$ and $[f,f')$ are denoted by $\mathrm{FL}([f,f'])$ and
         $\mathrm{FL}([f,f'))$, respectively. Then  $\mathrm{FL}([f,f'))$ is a subcomplex
         of $\mathrm{FL}([f,f'])$.
 \n
       For any face $f$ of $P$, let $b_f$ be the barycenter of $f$ (any point
        in the relative interior of $f$ would suffice).
       Then any flag $\alpha = (f_1 <\cdots <f_k) \in \mathrm{Fl}(\mathcal{P})$
       determines a unique simplex $\Delta_{\alpha}$ of dimension $|\alpha|-1$ 
       which is the convex hull of
        $b_{f_1}, \cdots, b_{f_k}$. It is clear that 
        the collection of simplices $\{ \Delta_{\alpha} \,|\, 
        \alpha \in \mathrm{Fl}(\mathcal{P})\}$ defines a triangulation of $P$, denoted by
        $\mathcal{T}(P)$, is called the
        \emph{barycentric subdivision} of $P$.
   \n
            
        For any $f\leq f' \in \mathcal{P}$, define
        \begin{equation} \label{equ:standard-cube}
         \square_{[f,f']} = \bigcup_{\alpha\in \mathrm{FL}([f,f'])} \Delta_{\alpha};  \ \
         \Delta_{[f,f')} =   \bigcup_{\alpha\in \mathrm{FL}([f,f'))} \Delta_{\alpha}.
         \end{equation}
          Note that $\Delta_{[f,f')}$ is simplicially isomorphic
          the barycentric subdivision of a simplex of dimension
          $\dim(f')-\dim(f) -1$. The complex $\square_{[f,f']}$ is the cone of 
          $\Delta_{[f,f')}$ with the point $b_{f'}$, i.e.
            $ \square_{[f,f']} = \mathrm{Cone}_{b_{f'}} (\Delta_{[f,f')})$.
          In fact $\square_{[f,f']}$ is simplicially isomorphic to the
       standard simplicial subdivision of the cube of dimension $\dim(f')-\dim(f)$.
         The \emph{standard cubical subdivision} of $P$, denoted by $\mathcal{C}(P)$,
          is the subdivision of $P$
         into $\{ \square_{[f,f']} \,|\, f \leq f' \in P\}$. \n
      
      For a small cover $M$ over an $n$-dimensional simple polytope $P$, the cubical subdivision of $P$ and the
       construction~\eqref{Equ:Quo-SC} of $M$ determines a
      cubical decomposition of $M$, denoted by $\mathcal{C}^s(M)$. The cubes in  
      $\mathcal{C}^s(M)$ are called \emph{small cubes} which are images of the cubes in the 
      $2^n$ copies of $P$ under the quotient map $\Theta$ in~\eqref{Equ:Theta}. 
         For any face $f$ of $P$, we define a family of \emph{big cubes} associated to $f$ 
         in $M$ by:
    \[ C^{(g)}_f = \bigcup_{h\in g+G_f} \Theta(\square_{[f,P]}, h), \ g\in (\Z_2)^n, f\in \mathcal{P}. \]
      The $G_f \subset (\Z_2)^n$ is defined in~\eqref{Equ:G_f}. 
      Let $\mathcal{C}(M)$ denote the set of all big cubes in $M$.
            $$ \mathcal{C}(M) = \{ C^{(g)}_f \,|\, f\in \mathcal{P},  g\in (\Z_2)^n\}.$$
       Notice that $ C^{(g)}_f=  C^{(g')}_f$ if
      $g - g'\in G_f$. So there are exactly $2^{\dim(f)}$ big cubes associated to the face $f$ in 
      $\mathcal{C}(M)$. In particular, there is only one big cube in
      $\mathcal{C}(M)$ associated to a vertex $v$ of $P$, denoted by $C_v$. 
      In addition, there are exactly 
       $2^n$ $0$-cubes in $\mathcal{C}(M)$ which are 
       $\{ C^{(g)}_P\,|\, g\in (\Z_2)^n \}$, where $P$ is considered as a face of itself.
         The reader is referred to~\cite[Sec.1.2]{DavJanScott98} or~\cite[Ch.4]{BP02} for more details of this construction.\n
       
          Since there are $2^n$ $0$-cells in $\mathcal{C}(M)$, 
          it is not so convenient for us to 
          write a presentation of $\pi_1(M)$ from $\mathcal{C}(M)$. 
          But notice that for any vertex $v$ of $P$, all the $0$-cells in $\mathcal{C}(M)$
          are contained in the big 
          $n$-cube $C_v$. In the rest of our paper,
           we identify $v$ with the center of $C_v$. 
          If we
          shrink $C_v$ to the point $v$,
          all other big cubes in $\mathcal{C}(M)$ will be deformed simultaneously to give us a new cell
          decomposition of $M$, denoted by $\mathcal{D}_v(M)$ (see Figure~\ref{p:Q_v}).
           Specifically, 
           for any face $f$ of $P$ and any $g\in (\Z_2)^n$, let $D^{(g)}_{f}$ denote the open cell in
           $\mathcal{D}_v(M)$ that comes from the interior of the big cube $C^{(g)}_f$. 
           So all the open cells in $\mathcal{D}_v(M)$ are
           \begin{equation}
            \mathcal{D}_v(M) =  \{ D^{(g)}_{f}\,|\, f \subset P, g\in (\Z_2)^n \}.
           \end{equation}
          It is clear that for any face $f$ of $P$, 
          $D^{(g)}_f=  D^{(g')}_f$ for any $g - g'\in G_f$, $g,g'\in (\Z_2)^n$.
          Let $\lambda$ be the characteristic function of $M$.
          Then the $2$-skeleton of $M$ with respect to $\mathcal{D}_v(M)$ consists 
          of the following cells.
           \begin{itemize}
            \item The only $0$-cell is $v$. 
                        \n
             \item The open $1$-cells are
            $\{  D^{(g)}_{F} \,|\, v\notin F\in \mathcal{F}(P), g\in (\Z_2)^n \}$, where
                $D^{(g)}_{F} = D^{(g+\lambda(F))}_{F}$ for any $g\in (\Z_2)^n$.
                Note that if a facet $F$ contains $v$, 
                then $C^{(g)}_F \subset C_v$ for all $g\in(\Z_2)^n$, which
                implies $D^{(g)}_{F} = v$.\n
           
         \item The open $2$-cells are
             $\{  D^{(g)}_{F \cap F'} \,|\, F \cap F' \neq \varnothing, F, F'\in \mathcal{F}(P),
                       g\in (\Z_2)^n \}$,   where 
              $D^{(g)}_{F \cap F'} =  D^{(g+\lambda(F))}_{F \cap F'} = 
              D^{(g+\lambda(F'))}_{F \cap F'} = D^{(g+\lambda(F)+\lambda(F'))}_{F \cap F'}$
              for any $g\in (\Z_2)^n$.
     \end{itemize}

         \n

         \begin{figure}
          % Requires \usepackage{graphicx}
         \includegraphics[width=0.93\textwidth]{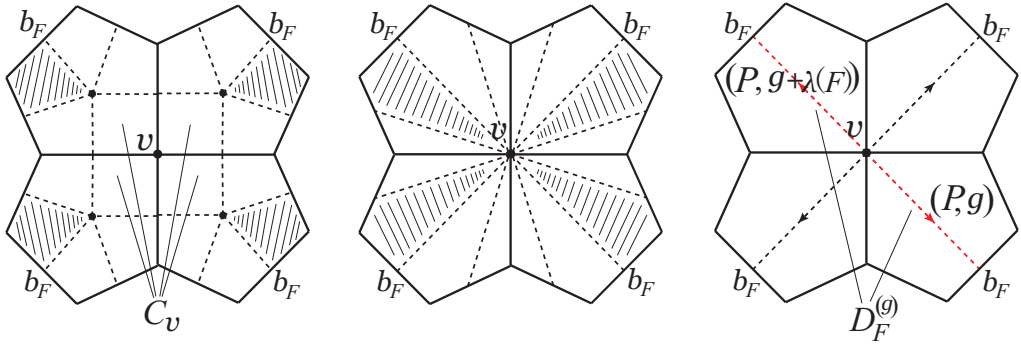}\\
          \caption{Shrinking the big cube $C_v$ to the point $v$}\label{p:Q_v}
      \end{figure}

    Since every $1$-cell $D^{(g)}_{F}$ in $\mathcal{D}_v(M)$ is attached to $v$, 
    $D^{(g)}_{F}$ along with an orientation determines a generator of the fundamental group of $M$.  
    In the polytope $P$, we orient the line segment 
     $\overline{vb_{F}}$ by going from $v$ to $b_{F}$, denoted by $\overrightarrow{vb_{F}}$
     \begin{itemize}
      \item For any facet $F$ not containing $v$, the closure of $D^{(g)}_{F}$ is the union of 
    $\Theta((\overline{vb_{F}},g))$ and $\Theta((\overline{vb_{F}}, g+\lambda(F))$
    as a set (see Figure~\ref{p:Q_v}). Considering the orientation, let
           $\beta_{F,g}$ denote the following oriented closed path based at $v$:
           $$\beta_{F,g} = \Theta((\overrightarrow{vb_{F}},g)) \cdot
            \Theta((\overrightarrow{vb_{F}}, g+\lambda(F))^{-1},\ g\in (\Z_2)^n.$$
           Then $\beta_{F,g}$ and $\beta_{F,g+\lambda(F)}$ correspond to the same
           $1$-cell $D^{(g)}_{F}$, but they have opposite orientations. So we have
        $\beta_{F,g+\lambda(F)} = \beta_{F,g}^{-1}$ for any $g\in (\Z_2)^n$.\n
     
     \item For any facet $F$ containing $v$, since $D^{(g)}_{F}=v$, 
       we let $\beta_{F,g}$ denote the trivial loop at the point $v$ 
       for any $g\in (\Z_2)^n$, and define $\beta_{F,g+\lambda(F)} = \beta_{F,g}^{-1}$. \n
     \end{itemize}

    We can consider $\{ \beta_{F,g}\,|\, F\in\mathcal{F}(P), g\in(\Z_2)^n \}$ 
     as a set of generators for $\pi_1(M,v)$. Moreover,
    any $2$-cell $D^{(g)}_{F \cap F'}$ in $\mathcal{D}_v(M)$ determines a relation
    $$\beta_{F,g}\beta_{F',g+\lambda(F)}=\beta_{F',g}\beta_{F,g+\lambda(F')}, \ g\in(\Z_2)^n.$$ 
    This can be easily seen from the picture of $D^{(g)}_{F \cap F'}$ in 
    Figure~\ref{p:Local-Model}.
      \n
               
         \begin{figure}
          % Requires \usepackage{graphicx}
         \includegraphics[width=0.97\textwidth]{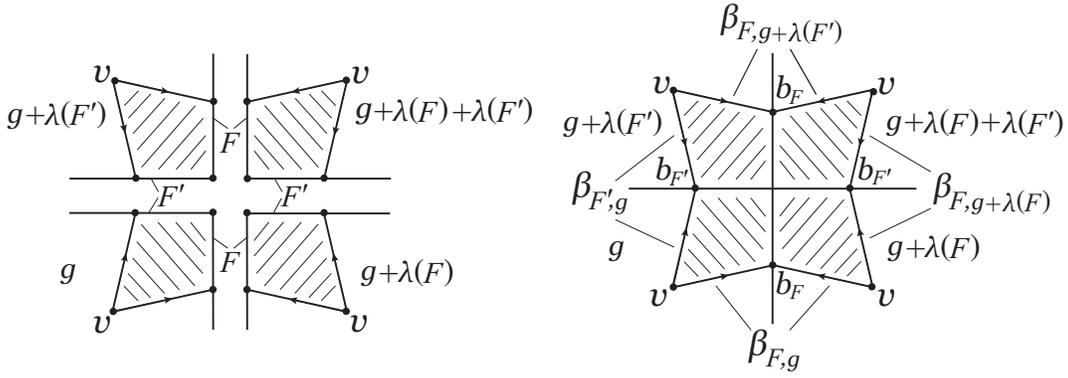}\\
          \caption{The cell $D^{(g)}_{F \cap F'}$}\label{p:Local-Model}
      \end{figure}
   
   For any facet $F$ of $P$, we define a transformation $\sigma_F: (\Z_2)^n \rightarrow (\Z_2)^n$ by:
      \begin{equation} \label{Equ:Sigma_F}
         \sigma_F(g) = g+\lambda(F), \ g\in (\Z_2)^n. 
       \end{equation}  
    Then we obtain the following proposition from the above discussion.
     
     \begin{prop} \label{Prop:Presentation}
      By the above notations, a presentation of $\pi_1(M,v)$ is given by   
      \begin{align}
     \big{\langle}  \beta_{F,g}, F\in \mathcal{F}(P), \, & g\in (\Z_2)^n \, |\, 
       \beta_{F,g}\beta_{F,\sigma_F(g)}=1, \forall g\in (\Z_2)^n; 
           \label{Relation-0} \\
       &\ \beta_{F,g}\beta_{F',\sigma_F(g)}=\beta_{F',g}\beta_{F,\sigma_{F'}(g)}, 
        F\cap F'\neq \varnothing, \forall g\in (\Z_2)^n; \label{Relation-1} \\
      & \ \beta_{F,g}=1,\, v\in F,  \forall g\in (\Z_2)^n \label{Relation-2}
           \big{\rangle}. 
\end{align}
\end{prop}
   Note that if $F\cap F'\neq \varnothing$ where $v\notin F$, $v\in F'$,
   the relation~\eqref{Relation-1} can be simplified to
   $\beta_{F,g} = \beta_{F,\sigma_{F'}(g)}$ for any $g\in (\Z_2)^n$.
  So if $v\notin F$ and $F'_{1},\cdots, F'_l$ are all the facets of $P$ which contains $v$ and 
  which intersect $F$ at the same time, we have
       \begin{equation}  \label{Relation-3}
      \beta_{F,g} = \beta_{F,g'}, \, \forall g, g'\in \langle \lambda(F'_{1}),\cdots,\lambda(F'_l) 
       \rangle \subset (\Z_2)^n.
   \end{equation}  
  If we remove the redundant generators 
  $\{\beta_{F,g}=1,\, v\in F, \forall g\in (\Z_2)^n\}$ 
  from the above presentation of $\pi_1(M,v)$, we obtain a
  simplified presentation of $\pi_1(M,v)$:
   \begin{align}
     \big{\langle}  \beta_{F,g},\, & v\notin F\in \mathcal{F}(P), g\in (\Z_2)^n \, |\, 
       \beta_{F,g}\beta_{F,\sigma_F(g)}=1, v\notin F, \forall g\in (\Z_2)^n; 
            \label{Sim-Presentation} \\
       &\  \beta_{F,g}\beta_{F',\sigma_F(g)}=\beta_{F',g}\beta_{F,\sigma_{F'}(g)},
       v\notin F, v\notin F',  F\cap F'\neq \varnothing, \forall g\in (\Z_2)^n ; 
       \notag \\
       &  \ \beta_{F,g} = \beta_{F,\sigma_{F'}(g)}, 
         v\notin F, v\in F', F\cap F' \neq \varnothing,
          \forall g\in (\Z_2)^n  \notag
       \big{\rangle}.   
\end{align}
  
  It is clear that the above presentation of the fundamental group of $M$ only depends on
  the choice of the vertex $v$ and the ordering of the facets of $P$.\nn
  
  \begin{rem}
   For any small cover $M$ over a simple polytope $P$,
  the minimal number of generators of $\pi_1(M)$ 
    is equal to $m-n$ where $n=\dim(M)$ and $m$ is the number of facets of $P$. 
    Indeed, the number of generators of $\pi_1(M)$ cannot be less than
   the first $\Z_2$-Betti number $b_1(M;\Z_2) = m-n$, since there is an epimorphism
   $$\pi_1(M)\rightarrow H_1(M)\rightarrow H_1(M;\Z_2).$$
    On the other hand, we can obtain a 
   cell decomposition of $M$ from the construction (I) with a single $0$-cell and exactly 
   $m-n$ $1$-cells.\n
   
   But the presentation of $\pi_1(M)$ with minimal number of generators is not so
    useful for our study of the $\pi_1$-injectivity problem of the facial submanifolds of $M$.
  \end{rem}
  \nn
 
   \subsection{Universal covering spaces of small covers}\ \n
   
     For any simple polytope $P$, define
      $$\mathscr{L}_P = P\times W_P\slash \sim$$
     where $(p,\omega)\sim (p',\omega')$ if and only if $p=p'$ and $\omega'\omega^{-1}$
      belongs the
     the subgroup of $W_P$ that is generated by 
     $\{ s_F \,|\, F \ \text{is any facet of $P$ that contains $p$} \}$.
     Indeed, if $F_1,\cdots, F_k$ are all facets of $P$ containing $p$, 
     then the subgroup of $W_P$ generated by $s_{F_1},\cdots, s_{F_k}$
     is isomorphic to $(\Z_2)^k$. This implies  
       \begin{equation} \label{Equ:Local-Equiv}
         (p,\omega)\sim (p,\omega')  \Longleftrightarrow
        \omega' = s^{\varepsilon_1}_{F_1}\cdots s^{\varepsilon_k}_{F_k} \cdot
         \omega,\
        \varepsilon_1,\cdots, \varepsilon_k \in \{ 0,1\}.
        \end{equation}
     
     There is a canonical action of 
     $W_P$ on $\mathscr{L}_P$ defined by:
     \begin{equation} \label{Equ:Action-L}
        \omega'\cdot [(p,\omega)] = [(p,\omega'\omega)], \ p\in P, \omega,\omega'\in W_P, 
      \end{equation}  
      where $[(p,\omega)]$ is the equivalence class of $(p,\omega)$ in $\mathscr{L}_P$.
      \n
      
     By~\cite[Corollary 10.2]{Davis83}, $\mathscr{L}_P$
     is a simply connected manifold. Moreover, $\mathscr{L}_P$
      is aspherical if and only if $P$ is a flag polytope 
      (see~\cite[Proposition 3.4.]{KurMasYu15}).\n
      
       Let $M$ be a small cover over $P$ with characteristic function $\lambda$.
     We have a homomorphism $\phi: W_P\rightarrow (\Z_2)^n$ where $n=\dim(P)$
     and $\ker(\phi)\cong \pi_1(M)$ (see~\eqref{Equ:Fund-Group}).
     
     \begin{lem} \label{Lem:phi-zero}
       Suppose $F_1,\cdots, F_k$ are facets of $P$ with 
     $F_1\cap \cdots \cap F_k \neq \varnothing$. Then
      $\phi(s^{\varepsilon_1}_{F_1}\cdots s^{\varepsilon_k}_{F_k}) \neq 0$
     as long as $\varepsilon_1,\cdots,\varepsilon_k \in \{0,1\}$ are not all $0$.
     \end{lem}
      \begin{proof}
        Without loss of generality, we can assume that
        $\lambda(F_1) = e_1,\cdots, \lambda(F_k) =e_k$ where
        $\{e_1,\cdots, e_n\}$ is a basis of $(\Z_2)^n$. So
        $\phi(s^{\varepsilon_1}_{F_1}\cdots s^{\varepsilon_k}_{F_k}) = \varepsilon_1e_1+
        \cdots + \varepsilon_k e_k \neq 0$ if
        $\varepsilon_1,\cdots,\varepsilon_k$ are not all $0$.
      \end{proof}
     
     The following proposition
      is contained in~\cite[Lemma 2.2.4]{DavJanScott98} (but with few 
      details for the proof). We give a proof here since
      some construction in the proof will be useful for our discussion later.\n
     
     \begin{prop} \label{prop:Iso-Action}
      The action of $\pi_1(M) \cong \ker(\phi) \subset W_P$ on $\mathscr{L}_P$
     through~\eqref{Equ:Action-L} is free with orbit space homeomorphic to
      $M$. So $\mathscr{L}_P$ is a universal covering of $M$.
     \end{prop}
     \begin{proof}
       For any $p\in P$, $\omega,\omega' \in W_P$,
        we see from~\eqref{Equ:Local-Equiv} and Lemma~\ref{Lem:phi-zero} that
       $$ (p,\omega) \sim (p, \omega'\omega) \Longrightarrow
               \omega' = 1 \ \text{or}\ \phi(\omega')\neq 0. $$
       So $(p,\omega) \nsim (p, \omega'\omega)$ for any
       $\omega' \neq 1 \in \ker(\phi)$. This implies that the action of
       $\ker(\phi)$ on $\mathscr{L}_P$ is free. \n
       
       Next, we show $\mathscr{L}_P\slash \ker(\phi)$ is homeomorphic to $M$.
       Choose a vertex $v$ of $P$ and let
        $F_1,\cdots, F_n$ be the $n$ facets of $P$ meeting $v$. Suppose all the facets
        of $P$ are $F_1,\cdots, F_n, F_{n+1},\cdots, F_m$.
       We can assume that $\lambda(F_i) =e_i$, $1\leq i \leq n$, where
        $\{e_1,\cdots, e_n\}$ is a basis of $(\Z_2)^n$. 
      For any $g\in (\Z_2)^n$, define
      \begin{equation}\label{Equ:gamma}
       \gamma_{g} = s_{F_{i_1}}\cdots s_{F_{i_l}} \in W_P, \ 
      g=e_{i_1} + \cdots + e_{i_l}\in (\Z_2)^n, 1\leq i_1 <\cdots < i_l \leq n.
      \end{equation}
     In particular, $ \gamma_{\lambda(F_i)} = \gamma_{e_i}=s_{F_i}$ and
      $\gamma_0 =1 \in W_P$.
     It is easy to see that
      \begin{equation} \label{Equ:gamma-property}
       \gamma_g^2 =1\in W_P , \ \ \phi(\gamma_g) =g, \ \forall g\in (\Z_2)^n.
      \end{equation}  
      So we have a monomorphism $\gamma :(\Z_2)^n \rightarrow W_P$ by mapping any
      $g\in(\Z_2)^n$ to $\gamma_g$. The image of $\gamma$ is the subgroup of $W_P$
      generated by $s_{F_1},\cdots, s_{F_n}$. Note that the definition of
       $\gamma$ depends on the vertex $v$ we choose.\n
      
      For any element $\omega\in W_P$, it is clear that
      $\gamma_{\phi(\omega)}\omega ^{-1} \in \ker(\phi)$.
      Define
      \begin{equation} \label{Equ:xi-omega}
         \xi_{\omega} = \gamma_{\phi(\omega)}\omega ^{-1} \in \ker(\phi) \subset W_P.  
      \end{equation} 
       The definition of
        $\xi_{\omega}$ depends on the
       choice of the vertex $v$ of $P$. By~\eqref{Equ:Action-L}, we have
       $$  \xi_{\omega} \cdot [(p,\omega)] = [(p,\gamma_{\phi(\omega)}], \ \forall p\in P,
         \omega\in W_P, $$
       So each orbit of the $\ker(\phi)$-action on $\mathscr{L}_P$ has a representative
       of the form $[(p,\gamma_g)]$
       where $g\in (\Z_2)^n$.\n
       
     Moreover, for any point $p$ in the interior of $P$,
      $[(p, \gamma_g)], [(p,\gamma_{g'})] \in \mathscr{L}_P$ are in the different
      orbits of the $\ker(\phi)$-action if $g\neq g'$. This is because
      for any $\omega\in \ker(\phi)$, we have
       $\omega\cdot [(p,\gamma_g)] = [(p, \omega \gamma_g)]$, so
       $\phi(\omega \gamma_g) = \phi(\gamma_g) = g \neq g' = \phi(\gamma_{g'})$.\n
      
       By the above discussion,
       a fundamental domain of the
      $\ker(\phi)$-action on the simply connected manifold $\mathscr{L}_P$ 
      can be taken to be the space $Q_v$ defined below.
      \begin{equation} \label{Equ:Q_v}
         Q_v = P\times \gamma((\Z_2)^n) \slash \sim 
       \end{equation}  
      where $(p,\gamma_g)\sim (p',\gamma_{g'})$ if and only if
       $p=p' \in F_1\cup \cdots \cup F_n$ and
       and $\gamma_{g'}\gamma_{g}^{-1}$ belongs to the subgroup of $\gamma((\Z_2)^n)$ spanned by
        $\{ s_{F_i} = \gamma_{\lambda(F_i)}  : p\in F_i, 1\leq i \leq n \}$.
        \n
       It is easy to see that
        $Q_v$ is the gluing of $2^n$ copies of $P$ by the same rule as~\eqref{Equ:Quo-SC} 
        along the facets $F_1,\cdots, F_n$
        while leaving other facets $F_{n+1},\cdots, F_m$ not glued.
        Furthermore, the gluing rule~\eqref{Equ:Local-Equiv} for $\mathscr{L}_P$ 
         tells us that the orbit space $\mathscr{L}\slash\ker(\phi)$
          is homeomorphic to the quotient of $Q_v$ by gluing 
         all the copies of $F_{n+1},\cdots, F_m$ on its boundary via the 
         rule in~\eqref{Equ:Quo-SC}, which is exactly the small cover $M$.
         So the proposition is proved.
     \end{proof}

    \nn
      
      \subsection{Isomorphism 
      from $\pi_1(M)$ to $\ker(\phi)$ defined by group presentations} \ \n

    For any small cover $M$ over a simple polytope 
    $P$, we use the presentation of $\pi_1(M)$ 
    in Proposition~\ref{Prop:Presentation}
    to construct an explicit isomorphism from
     $\pi_1(M)$ to $\ker(\phi)$ in this section. This isomorphism will be useful for us 
     to study the relations between the fundamental groups of $M$ and
     its facial submanifolds.\n
     
    Suppose all the facets of $P$ are $F_1,\cdots, F_n, F_{n+1},\cdots, F_m$ where
    $v=F_1 \cap \cdots \cap F_n$ is a vertex $P$. Then we have
      $$s_{F_i} s_{F_{i'}} = s_{F_{i'}} s_{F_i},\ 1\leq i,i'\leq n.$$ 
     We can assume that $\lambda(F_i)=e_i$, $1\leq i \leq n$ where $e_1,\cdots, e_n$ is 
     a basis of $(\Z_2)^n$. \n
     
    In the presentation of $\pi_1(M,v)$ in Proposition~\ref{Prop:Presentation},
    we let for brevity
     $$\beta_{j,g}=\beta_{F_j,g}, \ n+1\leq j \leq m.$$
     In addition, let  $\sigma_i = \sigma_{F_i}: (\Z_2)^n \rightarrow (\Z_2)^n$ where
      \begin{equation} \label{Equ:Sigma_i}
         \sigma_i(g) = g+\lambda(F_i), \ g\in (\Z_2)^n, 1\leq i \leq m. 
       \end{equation}  
      
      Then according to~\eqref{Sim-Presentation},
        a (simplified) presentation of $\pi_1(M,v)$ is given by:   
      \begin{align}
      \big{\langle}  \beta_{j,g}, n+1\leq j \leq m, g\in (\Z_2)^n\, & |\, 
       \beta_{j,g}\beta_{j,\sigma_j(g)}=1, n+1\leq j \leq m, \forall g\in (\Z_2)^n; 
           \notag \\
       & \ \beta_{j,g}\beta_{j',\sigma_j(g)}=\beta_{j',g}\beta_{j,\sigma_{j'}(g)},\, \forall g\in (\Z_2)^n\
          \text{where}
             \notag   \\
              & \ \ F_j\cap F_{j'}\neq \varnothing, n+1\leq j < j' \leq m; \notag  \\
        & \ \beta_{j,g} = \beta_{j,\sigma_{i}(g)}, \, \forall g\in (\Z_2)^n\ \text{where}  \notag \\
         &\ \  F_i\cap F_j\neq \varnothing, 1\leq i \leq n,  n+1\leq j \leq m 
         \big{\rangle} \label{Relation-2-Sim} 
    \end{align}

     On the other hand,
     we define a collection of elements of $W_P$ as follows.
     \begin{equation}\label{Equ:xi-def}
        \xi_{i,g} := \gamma_g s_{F_i} \gamma_{\sigma_i(g)} = 
     \gamma_g s_{F_i} \gamma_{g+\lambda(F_i)} \in W_P, \  1\leq i \leq m, 
      \forall g\in (\Z_2)^n.
     \end{equation} 
     Note that $\xi_{i,g} =1 \in W_P$ for any $1\leq i \leq n$ and $g\in (\Z_2)^n$.
     There are many relations between $\{ \xi_{j,g}\, |\, n+1\leq j \leq m,
     g\in (\Z_2)^n \}$ shown in the following lemma. \n

    \begin{lem} \label{Lem:xi}
     For any $g\in (\Z_2)^n$, we have
      \begin{itemize}
      \item[(i)] $\phi(\xi_{j,g})=0$, i.e. $\xi_{j,g} \in \ker(\phi)$,
        $n+1\leq j \leq m$;\n
      
      \item[(ii)]  $\gamma_{g} s_{F_{j}} = \xi_{j,g} \gamma_{\sigma_j(g)}$,
       $n+1\leq j \leq m$;\n
      
       \item[(iii)] $\xi_{j,g} \xi_{j,\sigma_j(g)}=1 \in W_P$,
        $n+1\leq j \leq m$;\n
       
       \item[(iv)] $\xi_{j,g}\xi_{j',\sigma_j(g)}=\xi_{j',g}\xi_{j,\sigma_{j'}(g)}$, 
        if $F_j\cap F_{j'}\neq \varnothing$ with 
          $n+1\leq j<j'\leq m$. \n
        
       \item[(v)] $\xi_{j,g} = \xi_{j,\sigma_i(g)}$, if $F_i\cap F_j\neq \varnothing, \,
              1\leq i \leq n, n+1\leq j \leq m$.
     \end{itemize}
     \end{lem}
     \begin{proof}
      It is easy to prove (i)(ii)(iii) from the definition of $\xi_{j,g}$.
     As for (iv), we have
       \begin{align*}
           \xi_{j,g}\xi_{j',\sigma_j(g)}&= \gamma_g s_{F_j}\gamma_{\sigma_j(g)} \cdot
          \gamma_{\sigma_j(g)} s_{F_{j'}} \gamma_{\sigma_{j'}(\sigma_j(g))} =
          \gamma_g s_{F_j} s_{F_{j'}} \gamma_{\sigma_{j'}(\sigma_j(g))};  \\
         \xi_{j',g}\xi_{j,\sigma_{j'}(g)} &= \gamma_g s_{F_{j'}} 
         \gamma_{\sigma_{j'}(g)} \cdot  \gamma_{\sigma_{j'}(g)} s_{F_j} 
          \gamma_{\sigma_{j}(\sigma_{j'}(g))} = \gamma_g s_{F_{j'}}  s_{F_j} 
          \gamma_{\sigma_{j}(\sigma_{j'}(g))} .
       \end{align*}      
     Note $\sigma_{j'}(\sigma_j(g)) = \sigma_{j}(\sigma_{j'}(g)) = g+\lambda(F_j)+\lambda(F_{j'})$, and
     $F_j\cap F_{j'}\neq \varnothing$ implies that
       $s_{F_j}s_{F_{j'}} = s_{F_{j'}} s_{F_j} \in W_P$. So we obtain
       $\xi_{j,g}\xi_{j',\sigma_j(g)}=\xi_{j',g}\xi_{j,\sigma_{j'}(g)}$.\n
       
        As for (v), since when $F_i\cap F_j\neq \varnothing$, $s_{F_i}$ commutes
      with $s_{F_j}$ in $W$. Then 
      $$ \xi_{j,\sigma_i(g)} = \gamma_{\sigma_i(g)} s_{F_j}
                                        \gamma_{\sigma_j(\sigma_i(g))} 
                             = \gamma_{g} s_{F_i} 
           s_{F_j} s_{F_i} \gamma_{\sigma_j(g)} = \gamma_g s_{F_j} \gamma_{\sigma_j(g)}
                   =\xi_{j,g}.$$ 
     So the lemma is proved.                     
 \end{proof}
    Note that (iv) and (v) in Lemma~\ref{Lem:xi} can be unified to the following form 
    \begin{equation} \label{Equ:Unified-Relation}
      \xi_{i,g}\xi_{j,\sigma_i(g)}=\xi_{j,g}\xi_{i,\sigma_{j}(g)}, \, 
       F_i\cap F_{j}\neq \varnothing, 1\leq i \leq m, 
          n+1\leq j\leq m. 
     \end{equation}  
    \begin{lem} \label{Lem:ker-phi}
      Any element of $\ker(\phi)$ can be written as a product of $\xi_{j,g}$.
    \end{lem}
    \begin{proof}
     Any element $\xi$ of $W_P$ can be written in the form
      $$\xi = \gamma_{g_1} s_{F_{j_1}} \gamma_{g_2} s_{F_{j_2}}\cdots \gamma_{g_k} s_{F_{j_k}} \gamma_{g_{k+1}}$$
       where $g_i\in (\Z_2)^n$ and $ n+1\leq j_i \leq m$ for all $i=1,\cdots, k+1$. 
       Note that we allow $\gamma_{g_i}$ to 
       be the identity of $W_P$ in the above expression of $\xi$.
        Then from the left to right, 
        we can replace the $\gamma_{g}s_{F_j}$ in the expression of $\xi$
      via the relation in Lemma~\ref{Lem:xi}(ii)
       until there is no $s_{F_j}$ in the end.
      So we can write 
      $\xi = \xi_{j_1, h_{1}} \xi_{j_2,h_2}\cdots \xi_{j_k,h_k} \gamma_{h_{k+1}}$
      for some $h_1,\cdots, h_{k+1}\in (\Z_2)^n$.
    Since $\xi_{j,g} \in \ker(\phi)$, $\phi(\xi) =0$ if and only if 
     $\phi(\gamma_{h_{k+1}})= h_{k+1} = 0$, i.e. $\gamma_{h_{k+1}}=1\in W_P$. 
     This proves the lemma.     
    \end{proof}
    
    Lemma~\ref{Lem:ker-phi} tells us that $\{ \xi_{j,g} \,|\, n+1\leq j \leq m, 
      g\in (\Z_2)^n \}$ form a set of generators for $\ker(\phi)$.
      To understand the relations between these generators, we adopt some ideas
      originated from Poincar\'e's study on discrete group of motions. 
      The content of the following paragraph is taken from~\cite[Part II, Ch.2]{Vinberg}.
       \n

      Suppose $Q$ is a normal fundamental polyhedron of the action of 
      a discrete group $\Gamma$ on a simply connected $n$-manifold $X$.
      The polyhedron $\gamma Q$, $\gamma\in \Gamma$, are called \emph{chambers}.
      For any $(n-1)$-dimensional face $F$ of $Q$, we denote by $a_F$ the element of
      $\Gamma$ taking $Q$ to the polyhedron adjacent to $Q$ along the face $F$, and
      by $F^*$ the inverse image of the face of $F$ under the action of $a_F$.
      Clearly we have
         \begin{equation}\label{Equ:Paring-Relation}
           a_F\cdot a_{F^*} = e
          \end{equation} 
      The action by $a_F$ is called an \emph{adjacency transformation}.
      Relations between the adjacency transformations are of the form
      $a_{F_1}\cdots a_{F_k}=e$ which correspond to the \emph{cycles of chambers}
      $$ Q_0=Q, \ \ Q_1=a_{F_1}Q, \ \ Q_2=a_{F_1}a_{F_2}Q,\cdots, \ \
      Q_k = a_{F_1}\cdots a_{F_k}Q =Q,$$
      where $Q_i$ and $Q_{i-1}$ are adjacent along the face 
      $a_{F_1}\cdots a_{F_{i-1}} F_i$.\n
      
     $\bullet$ Each $(n-1)$-dimensional face $F$ of $Q$ defines a cycle of chambers
      $(Q,a_F Q, Q)$, which corresponds
        to the relation~\eqref{Equ:Paring-Relation}, called a \emph{pairing relation}.\n
        
      $\bullet$ Each $(n-2)$-dimensional face determines a cycle consisting of
       all chambers containing this face in the order in which they are encountered 
       while circling this face. The corresponding relation is called a
        \emph{Poincar\'e relation}.

       \begin{thm}[see~\cite{Vinberg}] \label{thm:Poincare}
      The group $\Gamma$ is generated by adjacency transformations $\{a_F\}$. Moreover,
       the paring relations together with 
       Poincar\'e relations form a complete set of
       relations of the adjacency transformations in $\Gamma$.
    \end{thm} \n
    
    \begin{lem} \label{Lem:Relation-Xi}
     The Lemma~\ref{Lem:xi}\,$\mathrm{(iii)(iv)(v)}$ give a complete 
     set of relations for the generators $\{ \xi_{j,g} \,|\, n+1\leq j \leq m, 
      g\in (\Z_2)^n \}$ in $\ker(\phi)$.
    \end{lem}  
    \begin{proof}
     According to the proof of Proposition~\ref{prop:Iso-Action}, we can take the
      fundamental domain of the action of $\ker(\phi)$ on $\mathscr{L}_P$
     to be $Q_v $ (see~\eqref{Equ:Q_v}) 
     where $v=F_1\cap\cdots\cap F_n$.
    For any $g\in(\Z_2)^n$, let $F_{j,g} \subset \partial Q_v$ 
    denote the image of the 
    copy of $F_j$ in $(P,\gamma_g)$ for each $n+1 \leq j \leq m$.
     It is easy to see that $\xi_{j,g}$ maps $F_{j,\sigma_j(g)}$ to $F_{j,g}$. So
      the adjacency transformation defined by the facet
      $F_{j,g}$ of $Q_v$ is exactly $\xi_{j,g}$. Moreover, we have:
  
   \begin{itemize}
    \item The pairing relations for $\{\xi_{j,g}\}$
                  are exactly given in Lemma~\ref{Lem:xi}(iii) since
                   by definition $\xi_{j,g}$ maps $F_{j,\sigma_j(g)}$ to $F_{j,g}$
                   for any $g\in (\Z_2)^n$.\n
     
    \item The Poincar\'e relations for $\{\xi_{j,g}\}$ 
                 are exactly given by Lemma~\ref{Lem:xi}(iv)(v) (or equivalent the relations
                 in~\eqref{Equ:Unified-Relation}). This is because
                  for any $g\in (\Z_2)^n$ 
                  there are exactly four chambers of $\mathscr{L}_P$ around
                  the face $F_{i,g}\cap F_{j,g}\subset Q_v$ and we will meet
                  $F_{i,g}, F_{j,\sigma_i(g)}, F_{i,\sigma_{j}(g)}, F_{j,g}$
                  in order when circling around $F_{i,g}\cap F_{j,g}$
                   (see Figure~\ref{p:Domain}). 
                  
                   \n
    \end{itemize} 
    Then our lemma follows immediately from Theorem~\ref{thm:Poincare}.
    \end{proof}
   
 \begin{figure}
          % Requires \usepackage{graphicx}
         \includegraphics[width=0.4\textwidth]{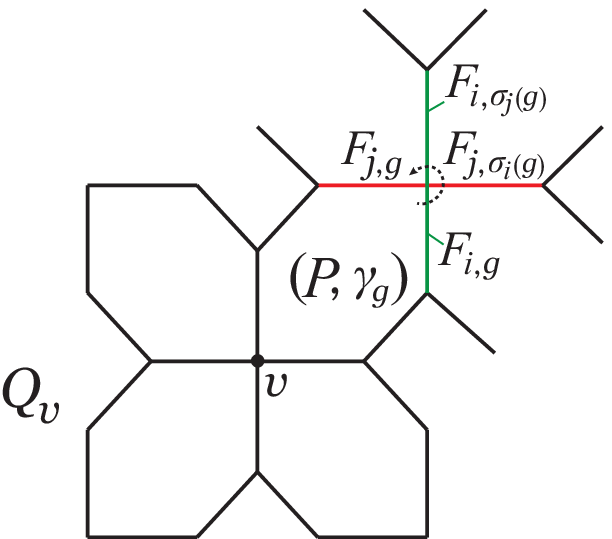}\\
          \caption{}\label{p:Domain}
      \end{figure}

    Now define a homomorphism $\Psi: \pi_1(M,v) \rightarrow W_P$ by
    \begin{equation} \label{Equ:alpha}
       \Psi(\beta_{j,g}) = \xi_{j,g},\ n+1\leq j \leq m, \forall g\in (\Z_2)^n.
     \end{equation}  
    By the presentation of $\pi_1(M,v)$ given in~\eqref{Relation-2-Sim} 
    and Lemma~\ref{Lem:Relation-Xi}, $\Psi$ is well defined and is an isomorphism
    from $\pi_1(M,v)$ to $\ker(\phi)$. So we obtain the following.\n
    
    \begin{lem} \label{Lem:Psi-Sequence}
     For any $n$-dimensional
      small cover $M$ over a simple polytope $P$, we have a short exact sequence 
      $1\longrightarrow \pi_1(M,v) \overset{\Psi}{\longrightarrow}
         W_P \overset{\phi}{\longrightarrow} (\Z_2)^n \longrightarrow 1$. 
    \end{lem}

   \begin{rem}
   The relation between $\Psi$ and $\psi$ (see~\eqref{Equ:Fund-Group}) is not very
    clear to us. The homomorphism $\Psi$ is defined via the presentation of $\pi_1(M)$ 
       while $\psi$ is induced by the map 
       $M \hookrightarrow M\times E(\Z_2)^n \rightarrow
         M_{(\Z_2)^n}$. But the
       isomorphism between $\pi_1(M_{(\Z_2)^n})$ and $W_P$ is not very explicit.
   \end{rem}
   
   Note that the sequence $1\longrightarrow \pi_1(M,v) \overset{\Psi}{\longrightarrow}
         W_P \overset{\phi}{\longrightarrow} (\Z_2)^n \longrightarrow 1$ splits since
         $\phi\circ \gamma = \mathrm{id}_{(\Z_2)^n}$ (see~\eqref{Equ:gamma-property}).
          Then $W_P \cong \pi_1(M,v)\rtimes (\Z_2)^n$ where
          $(\Z_2)^n$ acts on $\pi_1(M,v)$ by:
          $ g'\cdot \xi_{j,g} = \gamma^{-1}_{g'}\xi_{j,g}\gamma_{g'}=
          \xi_{j,g+g'} $ for any $g,g'\in (\Z_2)^n$ (see~\eqref{Equ:xi-def}).
          In other words, $(\Z_2)^n$ permutates the $2^n$ generators associated
          to each facet $F_j$ in the presentation~\eqref{Relation-2-Sim} 
          of $\pi_1(M,v)$.\nn
          
       \subsection{Torsion in the fundamental groups of small covers} 
       \label{subsec:Right-Coxeter}\ \n
       
      Given a finite simplicial complex $K$ with vertex set $\{1,\ldots, m\}$,
       there is a right-angled Coxeter group $W_K$ with generators  
       $s_1,\cdots, s_m$ and relations $s^2_i=1$, $1\leq i \leq m$
       and $(s_is_j)^2=1$ for each $1$-simplex $\{i,j\}$ in $K$.
       It is clear that 
       $$W_K = W_{K^{(1)}}$$
        where $K^{(1)}$ is the $1$-skeleton
       of $K$.        
        It is shown in~\cite[Sec 1.2]{Davis08} that 
       there is a finite cubical complex $\R \mathcal{Z}_K$ determined by $K$ 
       whose fundamental group is isomorphic to the commutator subgroup 
       $[W_K,W_K]$ of $W_K$. We call $\R \mathcal{Z}_K$
        the \emph{real moment-angle complex} of $K$. In some
        literature, $\R \mathcal{Z}_K$ is also denoted by $(D^1,S^0)^K$. 
         
         \begin{defi}[Real Moment-Angle Manifold] \label{Defi:Real-MAM}
           For an $n$-dimensional simple polytope $P$, $\R \mathcal{Z}_{\partial P^*}$ is a closed 
           connected $n$-manifold called the 
         \emph{real moment-angle manifold} of $P$, also denoted by $\R \mathcal{Z}_P$.
         The Coxeter group $W_{\partial P^*}$ clearly coincides with $W_P$ defined earlier. So $\pi_1(\R \mathcal{Z}_P) \cong [W_P,W_P]$. 
          Let the set of facets of $P$ be
    $\{ F_1,\cdots, F_m \}$. We can also obtain $\R \mathcal{Z}_P$ by gluing $2^m$ copies of $P$ via a function
         $\mu: \{ F_1,\cdots, F_m \}\rightarrow (\Z_2)^m$ 
          in the same way as~\eqref{Equ:Quo-SC} where $\{\mu(F_i)=e_i, 1\leq i \leq m\}$ forms a basis of $(\Z_2)^m$ (see~\cite[\S 4.1]{DaJan91} or~\cite{KurMasYu15}). 
      Let $$\Theta: P\times (\Z_2)^m \rightarrow \R \mathcal{Z}_P$$
       be the quotient map.
         There is a \emph{canonical $(\Z_2)^m$-action}
          on $\R \mathcal{Z}_P$ defined by
          \[  g'\cdot \Theta(x,g) = \Theta(x,g+g'),\ 
           x\in P^n, g,g'\in (\Z_2)^m,\]
 whose orbit space is $P$.  
 Let $\pi_P: \R \mathcal{Z}_P \rightarrow P$
         be the projection. For any proper face $f$ of $P$, it is easy to see that
          $\pi^{-1}_P(f)$ consists of
         $2^{m+\dim(f)-n-l}$ copies of $\R\mathcal{Z}_f$ where $l$ is the number of facets of $f$.
          Note that $\R \mathcal{Z}_P$ is always a closed connected orientable manifold.\n
            If there is a small cover $M$ over $P$, then there exists a subgroup 
          $H\cong (\Z_2)^{m-n}$ of $(\Z_2)^{m}$ where $H$ acts freely
           on $\R \mathcal{Z}_P$ (through the canonical action)
         whose oribt space is $M$. In other words, $\R \mathcal{Z}_P$ is a regular 
         $(\Z_2)^{m-n}$-covering space of $M$. \n
         \end{defi} 
        
       We call $K$ a \emph{flag complex} if any finite 
        set of vertices of $K$, 
        which are pairwise connected by edges, spans a simplex in $K$.
        Suggestively we think of a non-flag complex as having a minimal
         empty simplex of some dimension greater than $1$, i.e., a subcomplex equivalent to the boundary of a $k$-simplex that does not actually span a $k$-simplex.\n
        
         By~\cite[Proposition 1.2.3]{Davis08},
       $\R \mathcal{Z}_K$ is aspherical if and only if
         $K$ is a flag complex. So when $K$ is a flag complex, $\pi_1(\R \mathcal{Z}_K)$ is torsion-free.         
         For any finite simplicial complex $K$, the minimal
         flag simplicial complex that contains $K$ is called
         the \emph{flagification} of $K$, denoted by $\mathrm{fla}(K)$.
         Note that $K$ and $\mathrm{fla}(K)$ have the same $1$-skeleton. 
         So we have
        \begin{align*}
         [W_K,W_K] = [W_{K^{(1)}}, W_{K^{(1)}}] & =
          [W_{\mathrm{fla}(K)^{(1)}}, W_{\mathrm{fla}(K)^{(1)}}] \\
           & = [W_{\mathrm{fla}(K)},W_{\mathrm{fla}(K)}] \cong \pi_1(\R \mathcal{Z}_{\mathrm{fla}(K)}).
           \end{align*}
     This implies that $[W_K, W_K]$ is torsion-free for any finite simplicial complex $K$. Then since 
           $W_K\slash [W_K,W_K] \cong (\Z_2)^m$, it is easy to see that
          any torsion element of $W_K$ must have order $2$. 
          Then by~\eqref{Equ:Fund-Group}, we obtain the following.

     \begin{prop}
      Any element of the fundamental group of a small cover either has infinite order
       or has order $2$.
     \end{prop}  
    
    Note that right-angled Coxeter groups are special instances of a more general construction called
    \emph{graph products of groups}.
    The reader is referred to~\cite{PanVer16} for the study of the commutator subgroup of
    a general graph product of groups. 
   \vskip .6cm

   \section{$\pi_1$-injectivity of the 
             facial submanifolds of small covers} \label{Sec:Facial-SubMfd}
   
    Let $M$ be a small cover over a simple polytope $P$ with characteristic function $\lambda$.
     Given a proper face
      $f$ of $P$, we choose a vertex $v$ of $P$ contained in $f$. 
        Let all the facets of $P$ be $F_1,\cdots, F_m$ where 
        \begin{itemize}
          \item $F_1\cap \cdots\cap F_n=v$, $\lambda(F_i) =e_i$, $1\leq i \leq n$,
            $e_1,\cdots, e_n$ is a basis of $(\Z_2)^n$;\n
            
          \item $F_{1}\cap \cdots \cap F_{n-k} =f$, $\dim(f)=k$. So 
          $G_f=\langle e_1,\cdots, e_{n-k} \rangle \subset (\Z_2)^n$.\n
          \item $\mathcal{F}(f^{\perp}) = \{ F_{n-k+1}, \cdots, F_{n},\cdots, F_{n+r} \}$, $r\leq m-n$.
          So the codimension-one faces of $f$ are $\{ f\cap F_{n-k+1}, \cdots, f\cap F_{n},\cdots,
          f\cap F_{n+r}\}$ where the faces incident to $v$ are
          $f\cap F_{n-k+1}, \cdots, f\cap F_{n}$.
        \end{itemize}  
        
       The facial submanifold $M_f$
          is a small cover over the simple polytope $f$ whose characteristic function 
          $\lambda_f$ is given by~\eqref{Equ:Quotient-Color}. 
          Note that $\{ \lambda_f(f\cap F_i)\,|\, n-k+1\leq i \leq n \} $ is a basis of 
          $(\Z_2)^{\dim(f)}$. Then we can identify
          $(\Z_2)^{\dim(f)} \cong (\Z_2)^n\slash G_f$ with the subgroup 
          $ \langle e_{n-k+1},\cdots, e_n \rangle \subset (\Z_2)^n$ via a monomorphism
          \begin{align*} 
          \qquad\qquad \iota: (\Z_2)^{\dim(f)} &\longrightarrow (\Z_2)^n \\
                         \lambda_f(f\cap F_i) &\longmapsto  \lambda(F_i) =e_i,\ n-k+1 \leq i \leq n.
           \end{align*}
        We will always assume this identification in the rest of this section.
         Under this identification, we can write the function $\lambda_f$ as 
          $$\lambda_f : 
           \{ f\cap F_{n-k+1}, \cdots, f\cap F_{n},\cdots,
          f\cap F_{n+r}\} \rightarrow (\Z_2)^{\dim(f)} =\langle e_{n-k+1},\cdots, e_n \rangle \subset (\Z_2)^n$$
            where $\lambda_f(f\cap F_i) = e_i, n-k+1 \leq i \leq n$.
          Notice that $\lambda_f(f\cap F_j)$ and $\lambda(F_j)$ are not necessarily equal
           when $n+1\leq j \leq n+r$. But we have
          \begin{equation}
            \lambda_f(f\cap F_j) - \lambda(F_j) \in G_f,\, n+1\leq j \leq n+r.
          \end{equation}   
          
      In addition, parallelly to the transformation $\sigma_i: (\Z_2)^n\rightarrow (\Z_2)^n$
        in~\eqref{Equ:Sigma_i}, we define a set of
          transformations $\sigma^f_i : (\Z_2)^{\dim(f)} \rightarrow (\Z_2)^{\dim(f)}$ by
      \begin{equation}
         \sigma^f_i(g) = g+\lambda_f(f\cap F_i), \ g\in (\Z_2)^{\dim(f)},
         \ n-k+1\leq i \leq n+r. \
      \end{equation}
      Then we have
      \begin{equation} \label{Equ:tau-sigma}
              \sigma^f_i(g) - \sigma_i(g) =
              \lambda_f(f\cap F_i) - \lambda(F_i)  
              = \begin{cases}
                      0,   &   \text{ $n-k+1\leq i \leq n$; } \\
                       \in G_f    &   \text{ $n+1\leq i \leq n+r$.} \\
               \end{cases}
      \end{equation}

      By the presentation of $\pi_1(M,v)$ in Proposition~\ref{Prop:Presentation},
     we can similarly obtain a simplified presentation of $\pi_1(M_f,v)$.
     Note that the codimension-one faces of $f$ that are not incident to $v$ are $f\cap F_{n+1},\cdots,
     f\cap F_{n+r}$. We let $\beta^f_{j,g}$ denote the closed path of $M_f$ based at $v$ corresponding to
     $f\cap F_{j}$ for each $n+1\leq j\leq n+r$. Then according 
     to~\eqref{Sim-Presentation}, we have a presentation of $\pi_1(M_f,v)$:
    \begin{align}
    \pi_1(M_f,v) = \big{\langle}  \beta^f_{j,g}, &\,  n+1\leq j \leq n+r, g\in (\Z_2)^{\dim(f)} \, |\,
      \notag \\
       &  \beta^f_{j,g}\beta^f_{j,\sigma^f_j(g)}=1, n+1\leq j \leq n+r,\, \forall g\in (\Z_2)^{\dim(f)}; 
          \notag \\
       & \beta^f_{j,g}\beta^f_{j',\sigma^f_j(g)}=\beta^f_{j',g}\beta^f_{j,\sigma^f_{j'}(g)},\,
       \forall g\in (\Z_2)^{\dim(f)} \ \text{where} \notag  \\
             & \ F_j\cap F_{j'} \cap f \neq \varnothing, 
        n+1\leq j < j' \leq n+r; \notag \\
        & \beta^f_{j,g}= \beta^f_{j,\sigma^f_i(g)}, \, \forall g\in (\Z_2)^{\dim(f)}\ \text{where}\notag \\
        &\ n-k+1\leq i\leq n, 
                 n+1\leq j \leq n+r 
           \big{\rangle}. 
\end{align}

     \begin{figure}
          % Requires \usepackage{graphicx}
         \includegraphics[width=0.32\textwidth]{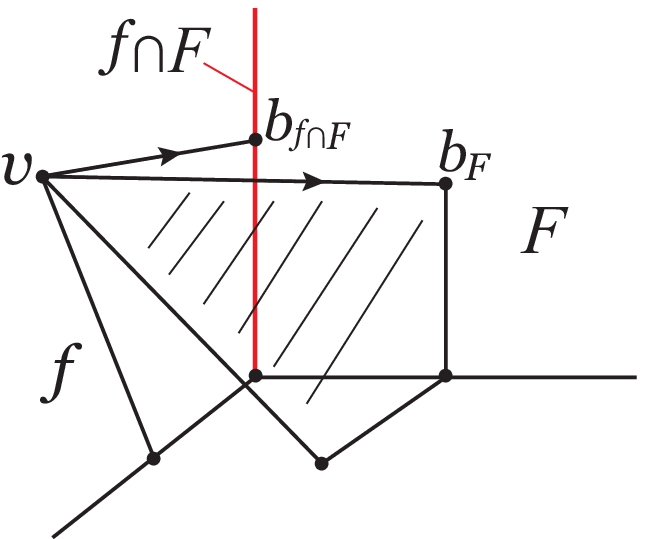}\\
          \caption{}\label{p:Inclusion}
      \end{figure} 
       
      \begin{lem} \label{Lem:j_f_star}
           The inclusion map $j_f : M_f \rightarrow M$ induces a homomorphism
           \begin{align} \label{Equ:j_f-star}
             (j_f)_* :\pi_1(M_f,v) & \longrightarrow \pi_1(M,v) \\ 
                   \beta^f_{j,g} \ &\longmapsto \ \beta_{j, g} ,  \ \forall g\in (\Z_2)^{\dim(f)}
                   \notag
            \end{align} 
      \end{lem}
      \begin{proof}
       By the definition of $\mathcal{D}_v(M)$, it is easy to see that
       the restriction of $\mathcal{D}_v(M)$ to $M_f$ 
       coincides with $\mathcal{D}_v(M_f)$. 
       Since $P$ is convex, we can deform the oriented line segment 
       $\overrightarrow{vb_{f\cap F}}$ to $\overrightarrow{vb_F}$ in $P$ with the endpoint $v$ fixed 
       (see Figure~\ref{p:Inclusion}).
       If we do this deformation in $\Theta((P,g))$ for all $g\in (\Z_2)^{\dim(f)}$ simultaneously
        in $M$, 
       we obtain a homotopy from the closed
        path $\beta^f_{j,g} \subset M_f$ to the closed path $\beta_{j,g} \subset M$
       which fixes the base point $v$
       for any $g\in (\Z_2)^{\dim(f)}$. This proves the lemma.
      \end{proof}

      We have the following constructions for any proper face $f$ of $P$.
          \begin{itemize}
           \item The inclusion map $i_f: f \rightarrow P$ induces a homomorphism
          $(i_f)_* : W_f \rightarrow W_P$ which sends the generator 
          $s_{f\cap F_i}$ of $W_f$
          to the generator $s_{F_i}$ of $W_P$ for any facet
           $F_i \in \mathcal{F}(f^{\perp})$.\n

          \item From the monomorphism $\gamma :(\Z_2)^n \rightarrow W_P$ defined in~\eqref{Equ:gamma},
            we obtain a monomorphism $\Psi: \pi_1(M,v)\rightarrow W_P$ in~\eqref{Equ:alpha}.
             For the face $f$, we similarly define a monomorphism
             $\gamma^f: (\Z_2)^{\dim(f)} \rightarrow W_f$ 
             by sending
             $\lambda_f(f\cap F_i) =e_i$ to $s_{f\cap F_i}$ for any $n-k+1\leq i \leq n$.
             Let $\gamma^f_g = \gamma^f(g)$ for any $g\in (\Z_2)^{\dim(f)}$. 
             Then since we identify $(\Z_2)^{\dim(f)}$ with 
             $\langle e_{n-k+1},\cdots, e_n \rangle \subset (\Z_2)^n$, we have
             \begin{equation} \label{Equ:gamma-tau}
                (i_f)_* ( \gamma^f_g ) =\gamma_{\iota(g)} =\gamma_g,
                 \forall g\in (\Z_2)^{\dim(f)}. 
              \end{equation}  
             So by Lemma~\ref{Lem:Psi-Sequence},
             we have a monomorphism $\Psi_f: \pi_1(M_f,v)\rightarrow W_f$ where
             \begin{equation} \label{Equ:Psi-f}
                \Psi_f(\beta^f_{j,g}) = \xi^f_{j,g} := \gamma^f_g s_{f\cap F_j}
                 \gamma^f_{\sigma^f_j(g)}. 
              \end{equation}  
             
           \item   Let $H_f$ be the normal subgroup of $W_P$ generated by the following set
    $$\{ s_F \,|\, f\subset F\} = \{ s_{F_1},\cdots, s_{F_{n-k}}\}.$$
     And let  $\eta_f: W_P \rightarrow  W_{P}\slash H_f$ be the projection. Then by definition,
     $\gamma$ maps $G_f \subset (\Z_2)^n$ isomorphically onto $H_f\subset W_P$.\n
     
     \item
    Let $\overline{\Psi} = \eta_f\circ \Psi$ and
     $\overline{(i_f)}_* = \eta_f\circ (i_f)_*$ be the compositions of $\Psi$ and 
     $(i_f)_*$ with $\eta_f$.    \n

       \end{itemize}

   \begin{lem} \label{Lem:Commute-Diagram}
     The following diagram is commutative.
     \begin{equation}\label{Equ:Commute-Diagram}
        \xymatrix{
       \pi_1(M_f,v) \ar[r]^{\ \ \Psi_f} \ar[d]^{(j_f)_*} & W_f 
      \ar[d]^{\overline{(i_f)}_*}  \\
       \pi_1(M,v)  \ar[r]^{\overline{\Psi}} & W_{P}\slash H_f    
      } 
      \end{equation}
   \end{lem}
   \begin{proof}
    For any $n+1\leq j \leq n+r$ and any $g\in (\Z_2)^{\dim(f)}$,
    $$ \overline{\Psi}\circ (j_f)_*(\beta^f_{j,g}) \overset{\eqref{Equ:j_f-star}}{=}
         \overline{\Psi}(\beta_{j,g}) = \eta_f (\Psi (\beta_{j,g})) 
         \overset{\eqref{Equ:xi-def}}{=}
     \eta_f(\xi_{j,g}) = \eta_f (\gamma_{g} s_{F_j} \gamma_{\sigma_j(g)}); $$
     $$\overline{(i_f)}_*\circ\Psi_f (\beta^f_{j,g})  \overset{\eqref{Equ:Psi-f}}{=}
       \overline{(i_f)}_*(\xi^f_{j,g})
     \overset{\eqref{Equ:Psi-f}}{=}
      \overline{(i_f)}_* ( \gamma^f_g s_{f\cap F_j} \gamma^f_{\sigma^f_j(g)} )
      \overset{\eqref{Equ:gamma-tau}}{=}  
      \eta_f( \gamma_{g} s_{F_j} \gamma_{\sigma^f_j(g)} ). $$
     
     By~\eqref{Equ:tau-sigma}, we have $\sigma^f_j(g) - \sigma_j(g) \in G_f$. So
     $\gamma_{\sigma^f_j(g)} \in \gamma_{\sigma_j(g)} \cdot H_f$. Then
       $\eta_f(\gamma_{\sigma^f_j(g)}) = \eta_f (\gamma_{\sigma_j(g)})$, which implies
     $ \overline{(i_f)}_*\circ\Psi_f (\beta^f_{j,g}) =  \overline{\Psi}\circ (j_f)_*(\beta^f_{j,g})$.  
     Then since $\pi_1(M_f)$ is generated by $\{ \beta^f_{j,g} \}$, the lemma is proved.   
   \end{proof}
   \n
   
   \begin{thm} \label{thm-Fund}
      Let $M$ be a small cover over a simple polytope $P$ and 
      $f$ be a proper face of $P$. Then the following two statements are equivalent.
   \begin{itemize}
    \item[(i)] The facial submanifold $M_f$ is $\pi_1$-injective in $M$.\n
     \item[(ii)] For any $F,F'\in \mathcal{F}(f^{\perp})$, we have
            $f\cap F\cap F' \neq \varnothing$ whenever $F\cap F' \neq \varnothing$.
   \end{itemize}         
  \end{thm} 
   \begin{proof}
     (i)$\Rightarrow$(ii). Using the conventions at the beginning of this section,
     we choose a vertex $v\in f$ and let 
     $f=F_1\cap\cdots \cap F_{n-k}$, $\mathcal{F}(f^{\perp}) = 
     \{ F_{n-k+1}, \cdots, F_{n},\cdots, F_{n+r} \}$.\n

     Assume that there exist two facets $F_j,F_{j'} \in \mathcal{F}(f^{\perp})$,
     $n+1\leq j < j' \leq n+r$ so that
      $f\cap F_j\cap F_{j'} = \varnothing$ while $F_j\cap F_{j'} \neq \varnothing$. 
      By the definition of $W_f$, the element
      $s_{f\cap F_j} s_{f\cap F_{j'}} \in W_f$ is of infinite order since
       $(f\cap F_j)\cap (f\cap F_{j'}) =\varnothing$. So we have
      $$(s_{f\cap F_j} s_{f\cap F_{j'}})^2 \neq 1 \in W_f; \ \
      (s_{F_j}s_{F_{j'}})^2=1 \in W_P.$$ 
      
      Moreover, by the definitions of $\Psi$ and $\Psi_f$ (see~\eqref{Equ:Psi-f}), 
      it is easy to check that
      \[ \mathbf{x}^f_{j,j'} := \beta^f_{j,0} \beta^f_{j',\lambda_f(f\cap F_j)}
            \beta^f_{j, \lambda_f(f\cap F_j) + \lambda_f(f\cap F_{j'})} 
              \beta^f_{j',\lambda_f(f\cap F_{j'})}  \overset{\Psi_f}{\longrightarrow}
         (s_{f\cap F_j} s_{f\cap F_{j'}})^2 \neq 1. \]
      \[ \mathbf{x}_{j,j'} := \beta_{j,0} \beta_{j',\lambda(F_j)} \beta_{j, \lambda(F_j) +
       \lambda(F_{j'})} \beta_{j',\lambda(F_{j'})}  \overset{\Psi}{\longrightarrow}
         (s_{F_j} s_{F_{j'}})^2 =1. \]   
     Then since $\Psi$ and $\Psi_f$ are both monomorphisms, we have
     $$ \mathbf{x}^f_{j,j'}\neq 1 \in \pi_1(M_f,v);\ \
      \mathbf{x}_{j,j'} =1\in \pi_1(M,v).$$ 
     The following claim is the heart of our argument.\n
     \textbf{Claim:} $(j_f)_*(\mathbf{x}^f_{j,j'}) = \mathbf{x}_{j,j'} =1\in \pi_1(M,v)$, i.e. 
      $\mathbf{x}^f_{j,j'} \in  \mathrm{ker}((j_f)_*)$.\n
      By Lemma~\ref{Lem:j_f_star}, $(j_f)_*(\mathbf{x}^f_{j,j'}) =
       \beta_{j,0} \beta_{j',\lambda_f(f\cap F_j)} \beta_{j, \lambda_f(f\cap F_j) +
       \lambda_f(f\cap F_{j'})} \beta_{j',\lambda_f(f\cap F_{j'})}$ which looks different
       from $\mathbf{x}_{j,j'}$.
       But since $v\in f = F_1\cap\cdots\cap F_{n-k}$ and
        $f\cap F_j\neq \varnothing$, $f\cap F_{j'}\neq \varnothing$, we have
         $F_i\cap F_j \neq \varnothing$ and $F_i\cap F_{j'} \neq \varnothing$ for all $1\leq i \leq n-k$. 
        Note that the subgroup of $(\Z_2)^n$ generated by $\lambda(F_1),\ldots, \lambda(F_{n-k})$
        is $G_f$ (see~\eqref{Equ:G_f}).
         Then according to~\eqref{Relation-3}, we have the
        following relations among the generators $\{ \beta_{j,g} \}$ and $\{ \beta_{j',g}\}$ in 
        the presentation~\eqref{Relation-2-Sim} of $\pi_1(M,v)$:
       $$ \beta_{j,g} = \beta_{j,h} ,\ \beta_{j',g} = \beta_{j',h}, \ \text{whenever}\ 
         g-h\in G_f, g, h\in(\Z_2)^n.$$
       Then since $\lambda_f(f\cap F_{j}) - \lambda(F_{j})\in G_f$, 
       $\lambda_f(f\cap F_{j'}) - \lambda(F_{j'})\in G_f$, we have
      \[   \beta_{j',\lambda_f(f\cap F_j)} = \beta_{j',\lambda(F_j)}; \ 
           \beta_{j, \lambda_f(f\cap F_j) + \lambda_f(f\cap F_{j'})} =   \beta_{j, \lambda(F_j) +
       \lambda(F_{j'})} ;\ 
        \beta_{j',\lambda_f(f\cap F_{j'})} = \beta_{j',\lambda(F_{j'})} .     \]
       It follows that $(j_f)_*(\mathbf{x}^f_{j,j'}) = 
        \mathbf{x}_{j,j'}$. The claim is proved.\n
          
       By the above claim, $\ker((j_f)_*)$ is not trivial, i.e. $(j_f)_*$ is not injective.       
     \nn
     
    (ii)$\Rightarrow$(i).
     By Lemma~\ref{Lem:Commute-Diagram}, we have $\overline{(i_f)}_*\circ\Psi_f
       = \overline{\Psi}\circ (j_f)_*$. We already know that $\Psi_f$ is injective
       from Lemma~\ref{Lem:Psi-Sequence}. So to prove
     the injectivity of $(j_f)_*$, it is sufficient to show that
      $\overline{(i_f)}_*$ is injective under the condition (ii). Note that      
      a presentation of $W_P\slash H_f$ is given by the generators 
      $\{ \bar{s}_F\}_{F\in \mathcal{F}(P)}$
       with relations 
       $$\{\bar{s}_F=1\,|\, f\subset F \in \mathcal{F}(P)\} \cup \{ \bar{s}_F^2=1 \}_{F\in \mathcal{F}(P)} \cup
        \{ (\bar{s}_F\bar{s}_{F'})^2 =1, \forall  F\cap F'\neq \varnothing\},$$
     where $\bar{s}_F =\eta_f(s_F)$ for any $F\in \mathcal{F}(P)$. Using this presentation 
     of $W_P\slash H_f$, we can define a homomorphism $\zeta_f : W_P\slash H_f \rightarrow W_f$ by 
      \begin{equation} \label{Equ:zeta_f}
           \zeta_f(\bar{s}_{F})  = \begin{cases}
                       s_{f\cap F},   &   \text{ $F\in \mathcal{F}(f^{\perp})$; } \\
                        1    &   \text{ $F \notin\mathcal{F}(f^{\perp})$ .} \\
               \end{cases} 
     \end{equation}   
       We claim that $\zeta_f$ is well defined. Indeed, for any $F,F'\in \mathcal{F}(f^{\perp})$ with $F\cap F' \neq \varnothing$, we have 
    $(\bar{s}_F\bar{s}_{F'})^2=1$ in $W_P\slash H_f$.
     Meanwhile, from the condition $f\cap F\cap F' \neq \varnothing$
    we have $(s_{f\cap F} s_{f\cap F'})^2=1$ in $W_f$. This implies that $\zeta_f$ is well defined.
    \n
    
    It is clear that $\overline{(i_f)}_*(s_{f\cap F}) = \bar{s}_F$ for any
    $F\in\mathcal{F}(f^{\perp})$. So 
    $\zeta_f \circ \overline{(i_f)}_* =\mathrm{id}_{W_f}$, which implies that
     $\overline{(i_f)}_*$ is injective. This finishes the proof. 
   \end{proof}
   
   \n
   
   By the proof of Theorem~\ref{thm-Fund}, we obtain
    the following description of the kernel of
   $(j_f)_* : \pi_1(M_f,v)\rightarrow \pi_1(M,v)$ for any proper face $f$ of 
   $P$ and any vertex $v\in f$.
   
   \begin{thm} \label{thm-Fund-Kernel}
      Let $M$ be a small cover over a simple polytope $P$ and 
      $f$ be a proper face of $P$. Let $j_f: M_f \rightarrow M$ be the inclusion map.
    Then for any vertex $v\in f$, the kernel of
     $(j_f)_*: \pi_1(M_f,v)\rightarrow \pi_1(M,v)$
     is the normal subgroup of $\pi_1(M_f,v)$ generated by 
    the inverse image of the following set under $\Psi_f: \pi_1(M_f,v)\rightarrow W_f$.
     \[ \Xi_f = \{ (s_{f\cap F} s_{f\cap F'})^2 \in W_f \,|\, v\notin F, v\notin F',
      f\cap F\cap F' = \varnothing\ 
       \text{while} \ F\cap F' \neq \varnothing \}. \]
   \end{thm} 
   \begin{proof}
    Let $I_f$ be the normal 
    subgroup of $\pi_1(M_f,v)$ generated by the set $\Psi_f^{-1}(\Xi_f)$.
    Let $J_f$ be the normal subgroup of $W_f$ generated by the set $\Xi_f$.
    \begin{itemize}
      \item By definition, $\Psi_f$ maps $I_f$ isomorphically onto $J_f$. 
      \n
      \item For any $(s_{f\cap F} s_{f\cap F'})^2 \in \Xi_f$, 
     the proof of (i)$\Rightarrow$(ii) in Theorem~\ref{thm-Fund} tells us that 
     $\Psi_f^{-1}((s_{f\cap F} s_{f\cap F'})^2) \in \ker((j_f)_*)$.
     So $I_f$ is contained in $\ker((j_f)_*)$.\n
     
     \item We clearly have $J_f \subset \ker((i_f)_*)$ and hence
           $J_f \subset \ker(\overline{(i_f)}_*)$. 
    \end{itemize}  
     Then we obtain the following commutative diagram from the
      diagram~\eqref{Equ:Commute-Diagram}.
       \begin{equation}\label{Equ:Commute-Diagram-quotient}
        \xymatrix{
       \pi_1(M_f,v)\slash I_f \ar[r]^{\ \ \ \, \Psi'_f} \ar[d]^{(j_f)'_*} & W_f\slash J_f 
      \ar[d]^{\overline{(i_f)}'_*}  \\
       \pi_1(M,v)  \ar[r]^{\ \overline{\Psi} } & W_{P}\slash H_f     
      } 
      \end{equation}
      where $(j_f)'_*$, $\Psi'_f$ and $\overline{(i_f)}'_*$ are the homomorphisms induced by
      $(j_f)_*$, $\Psi_f$ and $\overline{(i_f)}_*$ under the quotients by $I_f$ and $J_f$, respectively.\n
      
      So to prove $\ker((j_f)_*) = I_f$, it is equivalent to prove that $(j_f)'_*$ is injective.
       Note that $\Psi'_f$ is still injective. So by the 
       commutativity of the diagram~\eqref{Equ:Commute-Diagram-quotient},
        to prove $(j_f)'_*$ is injective, it is  
     sufficient to show that $\overline{(i_f)}'_*$ is
      injective.\n
      Indeed, the map $\zeta_f : W_P\slash H_f \rightarrow W_f$ 
       defined in~\eqref{Equ:zeta_f} naturally induces a 
       map $\zeta'_f : W_P\slash H_f \rightarrow W_f\slash J_f$.
        It is easy to check that $\zeta'_f$ is well defined and
        $\zeta'_f \circ \overline{(i_f)}'_* = \mathrm{id}_{W_f\slash J_f}$.
        This implies that $\overline{(i_f)}'_*$ is injective. So the theorem is proved.
   \end{proof}

     If a proper face $f$ of $P$ is a simplex with $\dim(f)\geq 2$, then for any $F,F'\in \mathcal{F}(f^{\perp})$,
     we have $f\cap F \cap F'=(f\cap F)\cap (f\cap F')$ is nonempty. So
      the condition in Theorem~\ref{thm-Fund} always holds for such a face $f$. 
        Then we obtain the following.
         
         \begin{cor} \label{Cor:simplex}
          Let $M$ be a small cover over a simple polytope $P$. 
            If a proper face $f$ of $P$ is a simplex
            and $\dim(f)\geq 2$, then
            the facial submanifold $M_f$ is $\pi_1$-injective in $M$.
         \end{cor}

          We can use Theorem~\ref{thm-Fund} to a description
          of aspherical small covers in terms of the $\pi_1$-injectivity 
          of their facial submanifolds.

\n
          
           \begin{prop} \label{Prop:Flag-Inject}
            A small cover $M$ over a simple polytope $P$
            is aspherical if and only if all the facial submanifolds of 
            $M$ are $\pi_1$-injective in $M$. 
           \end{prop}
           \begin{proof}
             If $M$ is aspherical, then $P$ is a flag polytope by~\cite[Theorem 2.2.5]{DavJanScott98}. Since any proper 
          face of a flag simple polytope clearly satsifies the
          condition (ii) in Theorem~\ref{thm-Fund}, it follows that
          all the facial submanifolds of $M$ 
           are $\pi_1$-injective in $M$. This proves the ``only if'' part
           of the proposition.\n
           
           Let $K_P= \partial P^*$ where $P^*$ is
           the dual polytope of $P$. Then $K_P$ is a simplicial sphere.
            To prove that $M$ is aspherical, it is equivalent to
           prove that $P$ is a flag polytope (i.e. $K_P$ is a flag simplicial
            complex) by~\cite[Theorem 2.2.5]{DavJanScott98}. 
            Furthermore, to prove that $K_P$ is flag, it is 
            equivalent to prove that for any simplex $\sigma$ of $K_P$ (including the empty simplex),
            the link $\mathrm{Lk}(\sigma,K_P)$ of $\sigma$ in $K_P$ has no 
            empty triangle (see~\cite[Lemma 3.4]{DavEdm14}). 
            For any proper 
            face $f$ of $P$, let $\sigma_f$ be the simplex in $K_P$ 
             corresponding to $f$. Note that $f$ is also a simple polytope and 
             $K_f =\partial f^*$ is isomorphic to $\mathrm{Lk}(\sigma_f, K_P)$
             as a simplex complex. \n
            
           If all the facial submanifolds of $M$ 
           are $\pi_1$-injective in $M$, then $P$ has no $3$-belts by
           Corollary~\ref{Cor:3-belt}. This implies that $K_P$
            contains no empty triangle (note that $K_P$ is the link of the empty
            simplex). Moreover, for any facet $F$ of $P$, all the facial submanifolds of $M_F$ must be $\pi_1$-injective in $M_F$. So by the 
            induction on the dimension of $M$, we can assume that $F$ is a 
            flag polytope, i.e. $K_F\cong \mathrm{Lk}(\sigma_F, K_P)$ is a flag complex. Note that $\sigma_F$ is a vertex of $K_P$.
            For any codimension-$k$ face $f$ of $P$ with $k\geq 2$, there exist $k$ distinct facets of $P$ so that $f=F_1\cap\cdots\cap F_k$. The vertex set of
             $\sigma_f$ is $\{\sigma_{F_1}, \cdots, \sigma_{F_k} \}$ and 
            $$\mathrm{Lk}(\sigma_f, K_P) = \bigcap^k_{i=1} 
            \mathrm{Lk}(\sigma_{F_i}, K_P). $$
            Assume that $\mathrm{Lk}(\sigma_f, K_P)$ contains an empty triangle with
            vertices $v_1,v_2,v_3$ (i.e. $v_1,v_2,v_3$ are pairwise connected by an edge in $\mathrm{Lk}(\sigma_f, K_P)$,  but they do no span a $2$-simplex in $\mathrm{Lk}(\sigma_f, K_P)$). Then since
            $\mathrm{Lk}(\sigma_{F_i}, K_P)$ is a flag complex, 
            $v_1,v_2,v_3$ must span a $2$-simplex in $\mathrm{Lk}(\sigma_{F_i}, K_P)$, denoted by $\tau_i$ for any $1\leq i \leq k$.
             But since any simplex in $K_P$ is determined uniquely by its vertex set, we must have $\tau_1=\cdots =\tau_k$. This implies that the $2$-simplex
             $\tau_i$ is actually
              contained in $\mathrm{Lk}(\sigma_f, K_P)$. This contradicts our assumption that $v_1,v_2,v_3$ do not span a $2$-simplex in $\mathrm{Lk}(\sigma_f, K_P)$. So $\mathrm{Lk}(\sigma_f, K_P)$ contains no
               empty triangle. Then we can conclude that $K_P$ is a flag complex
               by~\cite[Lemma 3.4]{DavEdm14},
               since for any nonempty 
               simplex $\sigma$ in $K_P$, there exists a unique face $f$
                with $\sigma=\sigma_f$.            
               This proves the ``if'' part of the proposition.
           \end{proof}

           \begin{rem} \label{Rem:Flag-Metric}
             There is another way to prove the ``only if'' part of
              Proposition~\ref{Prop:Flag-Inject} using ideas from
              metric geometry. 
              By~\cite[Theorem 2.2.5]{DavJanScott98},
           the natural piecewise Euclidean cubical metric $d_{\square}$
           on $M$ with respect to the small cubes decomposition $\mathcal{C}^s(M)$ 
           of $M$ is nonpositively curved if and only
             if $M$ is aspherical. Moreover,
             by~\cite[Corollary 1.7.3]{DavJanScott98} the inclusion of 
             the facial submanifold $M_f$ into 
             $(M,d_{\square})$ as a cubical subcomplex is a totally geodesic embedding.
             Then the ``only if'' part of Proposition~\ref{Prop:Flag-Inject} follows from these facts because
             any totally geodesic embedding into a non-positively curved metric
             space will induce a monomorphism in the fundamental group
            (see~\cite[Remark 1.7.4]{DavJanScott98}). But since the existence of a metric of nonpositive curvature is essential for this argument, this approach
             does not work when $P$ is not flag.       
           \end{rem}
      \n
            In Kirby's list~\cite{Kirby96} of problems in low-dimensional 
            topology, the Problem 4.119 
            asks one to find examples of aspherical $4$-manifolds with $\pi_1$-injective
            immersed $3$-manifolds with infinite fundamental group and wonders whether
             such kind of examples are
            common. If a small cover $M$ over a simple $4$-polytope $P$ 
               is aspherical, $P$ must be a flag polytope
             by~\cite[Theorem 2.2.5]{DavJanScott98}. Since any facet $F$ of $P$ is also 
             a flag polytope, $M_F$ 
             is a closed aspherical $3$-manifold 
             (hence with infinite fundamental group) and, $M_F$ is $\pi_1$-injective in $M$
              by Proposition~\ref{Prop:Flag-Inject}. So
             all $4$-dimensional
             aspherical small covers satisfy the condition in the Problem 4.119.
        
  \vskip .4cm
  
  \section{$3$-dimensional small covers and nonnegative scalar curvature} \label{Sec:ScalarCurvature}
    
    In this section, we will use the results obtained in
     Section~\ref{Sec:Facial-SubMfd} to study
    the scalar curvature of Riemannian metrics on $3$-dimensional small covers and real moment-angle manifolds.
    The reader is referred to~\cite{Peter06} for the basic notions of Riemannian geometry.
     Since every closed $3$-manifold has a unique smooth structure (Moise's theorem), we do not need to address the smooth structures here.
    First of all, we show that any $3$-dimensional small cover has a
    non-simply-connected embedded $\pi_1$-injective surface.

   \begin{prop} \label{Prop:Facial-3-dim}
     For any small cover $M$ over a $3$-dimensional simple polytope $P$, there
     always exists a facet $F$ of $P$ so that the facial submanifold $M_F$ is 
     $\pi_1$-injective in $M$.
   \end{prop}
   \begin{proof}
      By Corollary~\ref{Cor:3-belt}, the proposition is equivalent to saying that
      there always exists a facet $F$ of $P$ which is not contained in any $3$-belt on
      $P$. Otherwise assume that 
      every facet of $P$ is contained in some $3$-belt on 
      $P$. For the sake of brevity, we call a subset $D \subset \partial P$ 
      a \emph{regular disk} of $P$ if
      $D$ is the union of some facets of $P$ and 
      it is homeomorphic to the standard $2$-disk.\n
      
      Suppose $F_1,\cdots, F_m$ are all the facets of $P$ where
       $F_1, F_2, F_3$ form a $3$-belt on $P$. Then 
      the closure of $P\backslash (F_1\cup F_2\cup F_3)$ 
      consists of two disjoint regular disks of $P$, denoted by $D_1$ and $D'_1$. 
      By our assumption, any facet $F$ in $D_1$ 
      is contained in some $3$-belt on $P$. Clearly
      the $3$-belt containing $F$ can not contain any facet in $D'_1$.  
      This implies that there exists a $3$-belt in the regular disk 
       $\widetilde{D}_1 = D_1\cup F_1\cup F_2\cup F_3$ other than $(F_1, F_2, F_3)$. 
       Note that the number
      of facets in $\widetilde{D}_1$ is strictly less than $P$. 
      Next, we choose a $3$-belt $(F_{i_1},F_{i_2},F_{i_3}) \neq (F_1, F_2, F_3)$ 
      in $\widetilde{D}_1$.
      Then the closure of 
      $\widetilde{D}_1\backslash(F_{i_1} \cup F_{i_2} \cup F_{i_3})$  
      consists of two connected components, at
     least one of which is a regular disk of $P$, denoted by $D_2$.
      By our assumption any facet $F$ in $D_2$ is 
      contained in a $3$-belt in $P$ different from
       $(F_{i_1},F_{i_2},F_{i_3})$, which must lie
      in the regular disk
       $\widetilde{D}_2 = D_2\cup F_{i_1}\cup F_{i_2} \cup F_{i_3}$.
      By repeating this argument, we can obtain an infinite sequence of 
      regular disks $\{\widetilde{D}_i\}^{\infty}_{i=1}$ on $P$ where
      $\widetilde{D}_i\supsetneq \widetilde{D}_{i+1}$ for all $i\geq 1$ and 
      each  $\widetilde{D}_i$ contains a $3$-belt.
      But this is impossible since $P$ has only finitely many facets.
      The proposition is proved.     
   \end{proof}
   
   \begin{rem}
     For an even dimensional small cover $M$, it is possible that
     there are no codimension-one  $\pi_1$-injective facial submanifolds in $M$.
     Indeed, for any positive even integer $n$, let $P$ be the product of
      $n\slash 2$ copies of $2$-simplices $\Delta^2$ and let $M$ be any small cover $P$.
      For any $1\leq j \leq n\slash 2$, let
       $f^j_1,f^j_2,f^j_3$ be the three edges of the $j$-th copy of $\Delta^2$ in $P$. 
       Then all the facets of $P$ are
      $$\{ F^j_i = \Delta^2\times \cdots \times f^j_i \times \cdots \times \Delta^2 \,|\,
            1\leq j \leq n\slash 2, 1\leq i \leq 3\}.$$
      Clearly $(F^j_1, F^j_2, F^j_3)$ is a $3$-belt on $P$ for 
      any $1\leq j \leq n\slash 2$.
      So every facet of $P$ is contained in some $3$-belt on $P$. 
      Then by Corollary~\ref{Cor:3-belt}, $M$ 
      has no codimension-one $\pi_1$-injective facial submanifolds. But this kind of examples 
      seem quite rare. In addition, it is not clear to us whether there are such kind of small covers
       in any odd dimension greater than $3$.
   \end{rem}
   
   \begin{cor} \label{Cor:RMA-Dim3}
     For any $3$-dimensional simple polytope $P$, the real moment-angle manifold $\R \mathcal{Z}_P$ 
     has an embedded $\pi_1$-injective surface which is homeomorphic to $\R\mathcal{Z}_F$ for some facet $F$ of $P$. 
   \end{cor}
   \begin{proof}
      Suppose $P$ has $m$ facets. By Definition~\ref{Defi:Real-MAM},
       there is a 
      locally standard $(\Z_2)^m$-action on $\R \mathcal{Z}_P$
      with orbit space $P$. 
     By the Four Color Theorem, there exists a small cover $M$ over $P$ and then $\R \mathcal{Z}_P$ is a 
     regular $(\Z_2)^{m-3}$-covering space of $M$.  
     Let $\eta: \R \mathcal{Z}_P \rightarrow M$ be the covering map. 
     By Proposition~\ref{Prop:Facial-3-dim}, there exists a facet $F$ of $P$ so that the facial 
     submanifold $M_F$ is $\pi_1$-injective in $M$. Obviously
      $\eta^{-1}(M_F) = \pi^{-1}_P(F)$ is a disjoint union of copies of $\R\mathcal{Z}_F$, where
      $\pi_P: \R \mathcal{Z}_P \rightarrow P$ is the canonical projection.
    Let $j_F : M_F \rightarrow M$ be the inclusion and choose a
     copy of $\R\mathcal{Z}_F$ in $\eta^{-1}(M_F)$ and let $\widetilde{j}_F: \R\mathcal{Z}_F \rightarrow \R\mathcal{Z}_P $ be the inclusion.
       So $\eta$, $j_F$ and $\widetilde{j}_F$ induce a
      commutative diagram on the fundamental groups below. 
      \begin{equation*}
        \xymatrix{
       \pi_1(\R\mathcal{Z}_F) \ar[r]^{(\widetilde{j}_F)_*} \ar[d]^{\eta_*} & \pi_1(\R\mathcal{Z}_P) 
      \ar[d]^{\eta_*}  \\
       \pi_1(M_F)  \ar[r]^{\ (j_F)_*} & \pi_1(M)   
      } 
      \end{equation*}
      Then since $(j_F)_*$ and $\eta_*$ are all injective, so is $(\widetilde{j}_F)_*$. 
   \end{proof}
   
    By Schoen-Yau~\cite[Theorem 5.2]{SchYau79}, the existence of $\pi_1$-injective closed surfaces with nonpositive Euler characteristic
   in a compact orientable $3$-manifold $M$ is an obstruction to the existence of 
  Riemannian metric with positive scalar curvature on $M$.   
  Moreover, if $M$ admits a Riemannian metric with nonnegative scalar curvature and has an immersed $\pi_1$-injective orientable closed surface $S$ with positive genus, then $S$ must be
a torus and $M$ is Riemannian flat. 
 In addition, after works of Schoen, Yau, Gromov and Lawson, Perelman's proof of Thurston's the Geometrization Conjecture led to a complete classification of closed orientable $3$-manifolds which 
admit Riemannian metrics with positive scalar
curvature (see~\cite{Marq12}). These manifolds are the connected sum of spherical 3-manifolds and copies of $S^2 \times S^1$.
Combing these results with
Proposition~\ref{Prop:Facial-3-dim}, we can describe all the $3$-dimensional small covers and 
real moment-angle manifolds which can hold Riemannian metrics with positive or nonnegative scalar curvature below.\n

Let $P$ be an $n$-dimensional simple convex polytope in $\R^n$ and $v$ a vertex of $P$.  Choose a plane
$H$ in $\R^n$ such that $H$ separates $v$ from the other vertices of $P$. Let $H_{\geq}$ and $H_{\leq}$ be the two
half spaces determined by $H$ and assume that $v$ belongs to $H_{\geq}$. Then $P\cap H_{\geq}$ is an $(n-1)$-simplex, and
$P\cap H_{\leq}$ is a simple polytope, which we refer to as a \emph{vertex cut} of $P$. 
For example, a vertex cut of the $3$-simplex $\Delta^3$ is combinatorially equivalent to
 $\Delta^2\times [0,1]$ (the triangular prism).
 We use the notation $\mathrm{vc}^1(P)$ for any simple polytope that is obtained from $P$ by a vertex cut when the choice of the vertex is irrelevant. Note that up to combinatorial equivalence we can recover $P$ by shrinking the $(n-1)$-simplex $P\cap H_{\geq}$ on
$\mathrm{vc}^1(P)$ to a point.
 We also use the notation $\mathrm{vc}^k(P)$ for any simple polytope that is obtained from $P$ by iterating the vertex cut operation $k$ times for any $k\geq 0$ where $\mathrm{vc}^0(P)=P$. Note that for any simple polytop $P$ and any $k\geq 1$, $\mathrm{vc}^k(P)$
 is not a flag polytope since it always has a simplicial facet.
 
 \begin{rem}
   The simplicial polytope dual to any $\mathrm{vc}^k(\Delta^3)$
   is known as a \emph{stacked $3$-polytope}. A stacked $n$-polytope is a polytope obtained from 
   an $n$-simplex by repeatedly gluing another $n$-simplex onto one of its facets (see~\cite{MillReinStur07}).
 One reason for the significance of stacked polytopes is that, among all simplicial $n$-polytopes with a given number of vertices, the stacked polytopes have the fewest possible higher-dimensional faces. 
 \end{rem}

\begin{lem} \label{Lem:Vertex-Cut}
 Let $P$ be a $3$-dimensional simple polytope with $m$ facets. Then the real moment-angle manifold 
$\R\mathcal{Z}_{\mathrm{vc}^1(P)}$ is diffeomorphic to the connected sum of two copies of $\R\mathcal{Z}_P$ with
$2^{m-3}-1$ copies of $S^2\times S^1$. 
\end{lem}
\begin{proof}
  We can identify the $0$-sphere
  $S^0$ with $\Z_2$. Then by the definition of $\R\mathcal{Z}_{\mathrm{vc}^1(P)}$,
   $\R\mathcal{Z}_{\mathrm{vc}^1(P)}$ is produced from $\R\mathcal{Z}_P\times S^0$ by an equivariant surgery cutting $2^{m-3}$ balls from each copy 
   of $\R\mathcal{Z}_P$ and then connecting the boundary components
   by $2^{m-3}$ tubes $S^2\times S^1$ (see~\cite[\S 6.4]{BP02} 
   or~\cite[\S 2]{GitMed13}). 
\end{proof}

So indeed $\R\mathcal{Z}_{\mathrm{vc}^1(P)}$ is independent of where the vertex cut on $P$ is made. 
In general, $\R\mathcal{Z}_{\mathrm{vc}^k(P)}$ is a connected sum
of $2^k$ copies of $\R\mathcal{Z}_P$ with many copies of $S^2\times S^1$
for any $k\geq 0$. \n

\begin{defi}[Invariant Metric] \label{Defi:Inv-Metric}
 For a simple $3$-polytope $P$ with $m$ facets, a Riemannian metric $g$ on $\R \mathcal{Z}_P$ (or a small cover $M$ over $P$) is called an \emph{invariant metric} on $\R \mathcal{Z}_P$ (or $M$)
 if the canonical $(\Z_2)^m$-action (or canonical $(\Z_2)^3$-action)
  on $\R \mathcal{Z}_P$ (or $M$) is isometric with respect to $g$.
 Any invariant metric $g$ on $\R\mathcal{Z}_P$ can be projected to an invariant metric on $M$ which is locally isomorphic to $g$.  
  \end{defi}

  The following theorem tells us that 
 any equivariant surgery of codimension $\geq 3$ 
  can preserve the existence of
  invariant metric of positive scalar curvature on a manifold with respect to a compact Lie group action.

  \begin{thm}[cf. \cite{Ber83} Theorem 11.1] \label{Thm:Surgery}
 Let $M$ and $N$ be $G$-manifolds where $G$ is a compact Lie group.
Assume that $N$ admits an $G$-invariant metric of positive scalar curvature. If $M$
is obtained from $N$ by equivariant surgeries of codimension at least three, then $M$
admits a $G$-invariant metric of positive scalar curvature.
\end{thm}

\begin{prop} \label{Prop:ScalCur-RZ_P}
  Let $P$ be a $3$-dimensional simple polytope. The real moment-angle manifold
   $\R \mathcal{Z}_P$ admits a Riemannian metric with nonnegative scalar curvature if and only if
  $P$ is combinatorially equivalent to the cube $[0,1]^3$ or a polytope obtained from $\Delta^3$ by a sequence of vertex cuts. In this case $\R \mathcal{Z}_P$ is either the $3$-dimensional torus, the $3$-sphere or a connected sum of $(k-1)2^k+1$ copies of $S^2\times S^1$ for some $k\geq 1$, and the nonnegative scalar curvature metric on these 
  $\R \mathcal{Z}_P$ can be assume be to invariant.
\end{prop}
\begin{proof}
  Assume that $\R \mathcal{Z}_P$ admits a Riemannian metric with nonnegative scalar curvature.    
   \begin{itemize}
   \item[(i)] Suppose $P$ has no triangular facets. By Corollary~\ref{Cor:RMA-Dim3},
     there exists a facet $F\neq \Delta^2$ of $P$ so that $\R \mathcal{Z}_P$ 
     has an embedded $\pi_1$-injective surface homeomorphic to $\R\mathcal{Z}_F$.
    Then since $\R\mathcal{Z}_F$ is an orientable surface with positive genus,
     $\R \mathcal{Z}_P$ must be a flat manifold by~\cite[Theorem 5.2]{SchYau79}. 
     Then by~\cite[Theorem 1.2]{KurMasYu15},
    $P$ is combinatorially equivalent to the cube $[0,1]^3$.\n
     
    \item[(ii)]  Suppose $P$ has some triangular facets. Then unless $P$ is a $3$-simplex, we can
          shrink a triangular facet on $P$ to a point and obtain a new simple polytope with
           less facets than $P$ (see Figure~\ref{p:Triangle-Facets}). If $P$ is a $3$-simplex, the shrinking of its triangular facet to a point is invalid. Assume that
            after doing all possible (valid) shrinking of the
             triangular facets on $P$, we obtain a simple polytope $Q$ in the end. \n
          
          \begin{itemize}
           \item  If $Q$ is a $3$-simplex, then up to combinatorial equivalence $P$
            can be obtained from the $3$-simplex by a sequence of vertex cuts. \n
           
           \item If $Q$ is not a $3$-simplex, then $Q$ has no triangular facets.
            But we claim that this is impossible. 
                Indeed, 
                 $\R \mathcal{Z}_P$ is homeomorphic to the connected sum of 
                 copies of $\R \mathcal{Z}_Q $ with copies $S^2\times S^1$. 
                 But by Corollary~\ref{Cor:RMA-Dim3}, $\R \mathcal{Z}_Q $ has an embedded $\pi_1$-injective 
                 surface $\Sigma$ homeomorphic to $\R \mathcal{Z}_F$ for some facet $F$ of $Q$. 
                 Then since $F$ is not a triangle, $\Sigma\cong\R \mathcal{Z}_F$ is
                  a closed orientable surface with positive
                 genus. It is clear that the surface $\Sigma$ in each copy of 
                 $\R \mathcal{Z}_Q$ is also $\pi_1$-injective in $\R \mathcal{Z}_P$. 
                 Then by~\cite[Theorem 5.2]{SchYau79}, $\R \mathcal{Z}_P$ is a flat
                  manifold and hence $P$ is combinatorially equivalent to a $3$-cube
                 by case (i). This contradicts our assumption that $P$ has some 
                 triangular facets.                 
          \end{itemize} 
    \end{itemize}
    By the above discussion, $P$ is either combinatorially equivalent to $[0,1]^3$ or 
    a polytope obtained from $\Delta^3$ by a sequence of vertex cuts.\n
    
    Conversely, we know that $\R \mathcal{Z}_{[0,1]^3}$ is the $3$-dimensional 
   torus which admits an invariant flat Riemannian metric, $\R \mathcal{Z}_{\Delta^3}$ is the $3$-sphere
   which admits an invariant Riemannian metric with positive scalar (sectional) curvature. 
   By Lemma~\ref{Lem:Vertex-Cut}, 
    it is easy to see that
    $\R\mathcal{Z}_{\mathrm{vc}^k(\Delta^3)}$ is  diffeomorphic to the connected sum of
    $(k-1)2^k+1$ copies of $S^2\times S^1$.
    Moreover, by the proof of Lemma~\ref{Lem:Vertex-Cut}, 
    $\R\mathcal{Z}_{\mathrm{vc}^k(\Delta^3)}$
   is obtained from $\R \mathcal{Z}_{\Delta^3} \times (S^0)^k$ by a sequence
   of equivariant surgeries of codimension $3$. 
   So from Theorem~\ref{Thm:Surgery},
   we can conclude that $\R\mathcal{Z}_{\mathrm{vc}^k(\Delta^3)}$
    admits an invariant Riemannian metric with positive scalar curvature.
  The proposition is proved.
   \end{proof}

        \begin{figure}
          % Requires \usepackage{graphicx}
         \includegraphics[width=0.8\textwidth]{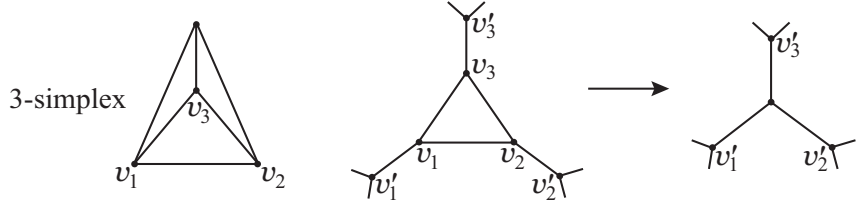}\\
          \caption{Shrinking a triangular facet to a point}\label{p:Triangle-Facets}
      \end{figure}

  It is shown in~\cite{KurMasYu15} that all $n$-dimensional
  small covers that admit flat Riemannian metrics are exactly the small covers over
  the $n$-cube $[0,1]^n$. These manifolds are called \emph{real Bott manifolds} in~\cite{MasKam09-1}. According to the discussion in~\cite[\S 7]{MasKam09-1}), there are exactly four diffeomorphism types among real Bott manifolds in dimension $3$,  
  two of which are orientable (the type $\mathcal{G}_1$ and $\mathcal{G}_2$ in the list of~\cite[Theorem 3.5.5]{Wolf77}). Moreover, Choi-Masuda-Oum~\cite{SuMasOum17} found
  an interesting combinatorial method to classify all real Bott manifolds up to affine diffeomorphism.

 \begin{prop} \label{Prop:ScalCur-SmallCover}
     A small cover $M$ over a simple $3$-polytope $P$ can hold a Riemannian metric with nonnegative
    scalar curvature if and only if $P$ is combinatorially equivalent to the cube $[0,1]^3$
   or a polytope obtained from $\Delta^3$ by a sequence of vertex cuts.
    In particular, all the orientable $3$-dimensional small covers that can hold Riemannian metrics with nonnegative scalar curvature are the two orientable real Bott manifolds in dimension $3$ and the
 connected sum of $k$ copies of $\R P^3$ for any $k\geq 1$.
 \end{prop}
 \begin{proof}
 Clearly, if $M$ admits a Riemannian metric of nonnegative scalar curvature, so is $\R\mathcal{Z}_P$. Conversely, if
  $\R\mathcal{Z}_P$ admits a Riemannian metric $g$ of nonnegative scalar curvature, Proposition~\ref{Prop:ScalCur-RZ_P} tells us
 that we can choose $g$ to be an invariant metric on $\R\mathcal{Z}_P$.
 Then $g$ projects to a metric on $M$
 which also has nonnegative scalar curvature (see Definition~\ref{Defi:Inv-Metric}). The first statement is proved.\n
   
   If $M$ is orientable, we can assume that the range of
     the characteristic function $\lambda$ of $M$ is in the subset $\{ e_1,e_2,e_3, e_1+e_2+e_3 \}$ of $(\Z_2)^3$ where $\{ e_1,e_2,e_3 \}$ is a basis of $(\Z_2)^3$ (see~\cite[Theorem 1.7]{NaNish05}). Then up to a change of basis, the value of $\lambda$ around a triangular facet is
     equivalent to the right picture in Figure~\ref{p:Vertex-Cut}.
     The surgery on the $3$-dimensional small cover $M$ corresponding to the vertex cut 
     in Figure~\ref{p:Vertex-Cut} is the connected sum of $M$ with $\R P^3$ (see~\cite[\S 5]{LuYu11}). Then since the small cover over $\Delta^3$
     is $\R P^3$, the second statement is proved.
 \end{proof}

   A closed Riemannian flat $3$-manifold cannot admit any Riemannian metric with positive scalar curvature (otherwise the $3$-torus would admit
   a metric of positive scalar curvature which is impossible). So we have the following corollary.

 \begin{cor} \label{Cor:ScalCur-SmallCover-2}
 A small cover or the real moment-angle manifold over a simple polytope $P$ admits a Riemannian metric with positive scalar curvature if and only if
 $P$ is combinatorially equivalent to a polytope obtained from $\Delta^3$ by a sequence of vertex cuts. 
 \end{cor}

 \begin{figure}
          % Requires \usepackage{graphicx}
         \includegraphics[width=0.52\textwidth]{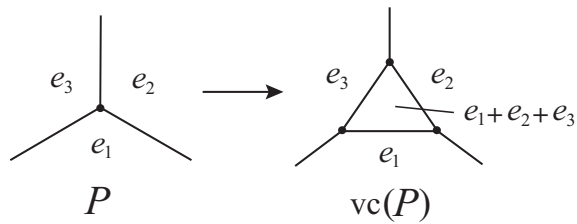}\\
          \caption{Vertex cut}\label{p:Vertex-Cut}
      \end{figure}

 We summarize
  results in this section along with some results from~\cite[\S 5]{KurMasYu15} in the Table 1 below, where we 
 list all the $3$-dimensional 
 simple polytopes and their dual simplicial polytopes over which
 the small covers (or real moment-angle manifolds) can hold 
 Riemannian metrics with various curvature conditions.
  
\begin{table}[!hbp] \renewcommand{\arraystretch}{1.5} \small 

\begin{tabular}{|c|c|c|} 

\hline

\hline

\small{$3$-dim. Small Cover} &  Simple $3$-Polytope & (Dual) Simplicial $3$-Polytope    \\

\hline

 Sectional/Ricci $>0$     &  $\Delta^3$  & $\Delta^3$ \\

\hline

 Scalar $>0$     &  $\mathrm{vc}^k(\Delta^3)$, $k\geq 0$    & Stacked $3$-polytopes  \\

\hline

 Sectional/Ricci $\geq 0$  & $[0,1]^3$, $\Delta^3$, $\Delta^2\times [0,1]$  & Octahedron, $\Delta^3$, Triangular bipyramid  \\

\hline

Scalar $ \geq 0$ &  $[0,1]^3$, $\mathrm{vc}^k(\Delta^3)$, $k\geq 0$ & Octahedron, Stacked $3$-polytopes \\
\hline    

\end{tabular}
\vskip .3cm
\caption{$3$-dimensional small covers which admit
  Riemannian metrics with various curvature conditions}
\end{table}

 In dimension $\geq 4$, characterizing all small covers and real moment-angle manifolds that admit Riemannian metrics of
 positive (or nonnegative) scalar curvature should be much harder than dimension $3$.
  Indeed, in dimension $4$ there are obstructions to the existence of metrics of 
       positive scalar curvature from Seiberg-Witten
       invariants.
 In dimension $\geq 5$, the existence of Riemannian metrics with positive 
    scalar curvature is intimately related to the existence of spin structures (see~\cite{GromLaw80}) and some   differential topological obstructions (e.g. $\widehat{A}$-genus) from index theory of the Dirac operator 
    (see~\cite{Hitch74, Lich63, Stolz92}). The reader is referred 
    to~\cite{Stolz02, Ros07} for a survey of this subject. But
    translating these conditions on small covers 
    into conditions on the underlying simple polytopes seems quite
    difficult. In addition, the minimal hypersurface method 
    (see~\cite{SchYau79-2, Schick98})
   should be useful for us to study this problem as well.

 \vskip .4cm
 
\section{Acknowledgment}
  The authors
   want to thank Taras Panov for some helpful discussion on the properties of
   right-angled Coxeter groups shown in section~\ref{subsec:Right-Coxeter} and,
   thank Li Cai for inspiring discussion on the proof of Proposition~\ref{Prop:Flag-Inject}.
   

\vskip .6cm


\begin{thebibliography}{99}


\bibitem{Ber83}
 L.~B\'erard Bergery, \emph{Scalar curvature and isometry groups}, 
  in Spectra of Riemannian manifolds, in Spectra of Riemannian Manifolds,
    Kagai Publications, Tokyo (1983), 9--28.

\bibitem{BjorBre05}
 A.~Bjorner and F.~Brenti, \emph{Combinatorics of Coxeter Groups},
  Graduate Texts in Mathematics 231, Springer Berlin Heidelberg, 2005.

 \bibitem{BP02} 
  V. M. Buchstaber and T.E. Panov,
 \emph{Torus actions and their applications in topology and
 combinatorics}, University Lecture Series, \textbf{24}.
 American Mathematical Society, Providence, RI, 2002.  


\bibitem{CGKZ98}
  D.~J.~Collins, R.~I.~Grigorchuk, P.~F.~Kurchanov and H.~Zieschang, 
  \emph{Combinatorial group theory and applications to geometry}, Springer 1998. 
  
 
\bibitem{SuMasOum17}
   S.~Choi, M.~Masuda and S.~Oum, \textit{Classification of
    real Bott manifolds and acyclic digraphs}, Trans. Amer. Math. Soc. 369 (2017), 2987--3011. 
   
\bibitem{Davis83}
 M.~W.~Davis, 
 \emph{Groups generated by reflections and aspherical manifolds not covered by
Euclidean space}, Ann. of Math. 117 (1983), 293--324.   
  
\bibitem{DaJan91}  M.~Davis and T.~Januszkiewicz, \textit{Convex polytopes,
Coxeter orbifolds and torus actions}, Duke Math. J. \textbf{62}
(1991), no.\textbf{2}, 417--451.

\bibitem{DavJanScott98}
 M.~Davis, T.~Januszkiewicz and R.~Scott, \textit{Nonpositive curvature of
 blow-ups}, Selecta Math.~4 (1998), no.~\textbf{4}, 491--547.
 
\bibitem{Davis08}
 M.~Davis, \emph{The geometry and topology of Coxeter groups}, 
  London Mathematical Society Monographs Series, vol. 32, Princeton University Press, 2008. 
  
\bibitem{DavEdm14}   
 M.~Davis and A.~Edmonds, \emph{Euler characteristics of generalized Haken manifolds},  Algebr. Geom. Topol. 14 (2014), 3701--3716.
 
 
\bibitem{GitMed13}
 S.~Gitler and S.~L\'opez de Medrano, \emph{Intersections of Quadrics, Moment-angle Manifolds and Connected Sums}, Geom. Topol. 17 (2013), no.~\textbf{3}, 1497--1534.

 
\bibitem{GromLaw80}
  M.~Gromov and H.~B.~Lawson, \emph{The classification of simply connected
manifolds of positive scalar curvature}, Ann. Math. (2), 111 (1980),
423--434. 


\bibitem{Hitch74}
 N.~Hitchin, \emph{Harmonic spinors}, Advances in Math. 14 (1974), 1--55.


\bibitem{Lich63}
 A.~Lichnerowicz, \emph{Spineurs harmoniques}, C.R. Acad. Sci. Paris, 257 (1963), 7--9.

\bibitem{MasKam09-1}
  Y.~Kamishima and M.~Masuda,
  \textit{Cohomological rigidity of real Bott manifolds}, Algebr. Geom. Topol. 9 (2009),
  2479-2502.

\bibitem{Kirby96}
 R.~Kirby, \emph{Problems in low-dimensional topology}, Georgia topology conference, 1996.

\bibitem{KurMasYu15}   
S.~Kuroki, M.~Masuda and L.~Yu, \emph{Small covers, infra-solvmanifolds and curvature}, 
Forum Math. 27 (2015), no.~\textbf{5}, 2981--3004.  


\bibitem{LuYu11}
 Zhi $\mathrm{L\ddot{u}}$ and Li Yu, \emph{Topological types of $3$-dimensional small covers}, 
  Forum Math. 23 (2011), no.~\textbf{2}, 245--284.    

 \bibitem{Marq12}
  F.C.~Marques, \emph{Deforming three-manifolds with positive scalar curvature}, Annals of
Mathematics 176 (2012), 815--863.

  
\bibitem{MillReinStur07}
E.~Miller, V.~Reiner and B.~Sturmfels, \emph{Geometric Combinatorics}, 
 IAS/Park City mathematics series 13 (2007), American Mathematical Society. 
  

\bibitem{NaNish05}
H.~Nakayama and Y.~Nishimura, \textit{The orientability of small
covers and coloring simple polytopes}, Osaka J. Math. 42 (2005),
243-256.


\bibitem{PanVer16}
T.~Panov and Y.~Veryovkin, \emph{Polyhedral products and commutator subgroups of right-angled Artin and Coxeter groups}, Sbornik Math. 207 (2016), no.~\textbf{11}, 1582--1600.

\bibitem{Peter06}
P.~Petersen, \emph{Riemannian Geometry}, (2006), Springer-Verlag, Berlin.

\bibitem{Ros07}
 J.~Rosenberg, \emph{Manifolds of positive scalar curvature: a progress report},
   Surveys in differential geometry. Vol. XI, 259--294, International Press, 2007.  

\bibitem{Schick98}
 T.~Schick, \emph{A counterexample to the (unstable) Gromov-Lawson-Rosenberg conjecture},
Topology 37 (1998), no.~\textbf{6}, 1165--1168.     

\bibitem{SchYau79}
R.~Schoen and S.~T.~Yau, \emph{Existence of incompressible surfaces and the topology of 3-dimensional
manifolds with non-negative scalar curvature}, Annals of Math. 119 (1979), 127--142.

\bibitem{SchYau79-2}
 R.~Schoen and S.~T.~Yau, \emph{On the structure of manifolds with positive scalar curvature}, Manuscripta Math. 28 (1979), no. \textbf{1-3}, 159--183.


\bibitem{Stolz92} 
S.~Stolz, \emph{Simply connected manifolds of positive scalar curvature}, 
Ann. of Math. (2), 136 (1992), no.~\textbf{3}, 511--540. 

\bibitem{Stolz02}
  S.~Stolz, \emph{Manifolds of positive scalar curvature},
   Topology of high-dimensional manifolds, No. 1, 2 (Trieste, 2001), 661--709, ICTP Lect. Notes, 9, Abdus Salam Int. Cent. Theoret. Phys., Trieste, 2002.

\bibitem{Vinberg}
    E.~B.~Vinberg and O.~V.~Shvartsman,
     \textit{Discrete groups of motions of spaces of contant curvature}, Geometry, II,
    139-248, Encyclopaedia Math. Sci., 29, Springer, Berlin, 1993.

\bibitem{Waldhau68}
 F.~Waldhausen, \emph{On irreducible $3$-manifolds which are sufficiently large},
   Ann. Math. 87, 56--88 (1968).
    
 
\bibitem{Wolf77}
 J.~A.~Wolf, \emph{Spaces of Constant Curvature, fourth edition},
  Publish or Perish, 1977. 
 
\bibitem{Ziegler95}
 G.~M.~Ziegler, Lectures on polytopes.
  Graduate Texts in Mathematics, 152 (revised first edition), Springer-Verlag, New York, 1998.
  
\end{thebibliography}
\end{document}